\documentclass[12pt]{article}
\usepackage[percent]{overpic}
\usepackage{tikz}
\usetikzlibrary{arrows, shapes, trees}
\graphicspath{{images/}}
\usepackage{color}
\everymath{\displaystyle}

\usepackage{amsmath, amsfonts, amsthm, amssymb, graphicx,  enumerate, fullpage, color, multicol, caption}

\newtheorem{definition}{Definition} 

\newtheorem{theorem}{Theorem}
\newtheorem{example}{Example}

\usepackage{authblk}

\title{Group Actions on Winning Games of Super Tic-Tac-Toe}

\author{Whitney George}
\affil{Department of Mathematics and Statistics\\ University of Wisconsin-La Cross\\ wgeorge@uwlax.edu}

\author{Janine E. Janoski}
\affil{Department of Mathematics and Computer Science\\ King's College \\ janinejanoski@kings.edu}

\date{}

\begin{document}

\maketitle

\begin{abstract}
Consider a $n \times n$ tic-tac-toe board. In each field of the board, draw a smaller $n\times n$ tic-tac-toe board. Now let super tic-tac-toe (STTT) be a game where each player's move dictates which field on the larger board  a player must make their next move. We will play an impartial game of STTT where each player uses X. We define a set of actions on a game board which gives rise to a group-action on the game that creates equivalent games. We will discuss how the structure of this group-action forms a Dihedral group.
\end{abstract}

\section{Introduction}
Tic-tac-toe is a two player game where players X and O take turns marking a $3\times 3$ grid. The player who successfully places three of their symbol in a horizontal, vertical, or diagonal row wins the game. The game of tic-tac-toe can be traced back to Ancient Egypt. 

Alternate variations of tic-tac-toe have been developed over time, such as 3 dimensional tic-tac-toe  on a $3\times 3 \times 3$ cube and Qubic played on a $4 \times 4 \times 4$ matrix sold by Parker Brothers. Other versions include quantum tic-tac-toe developed by Allan Goff \cite{goff} and tic-tac-toe played on an affine plane \cite{carroll}.

We will focus on the game of super tic-tac-toe (STTT), sometimes called ultimate tic-tac-toe. Super tic-tac-toe has been used in the computer science community as an example of Monte-Carlo Artificial Intelligence. Suppose we have a standard $3\times 3$ tic-tac-toe board. Now, inside each of the 9 squares, we place {\it another} tic-tac-toe board (Figure \ref{3x3}). 

\begin{figure}[h!]
\begin{center}\includegraphics[scale=.41]{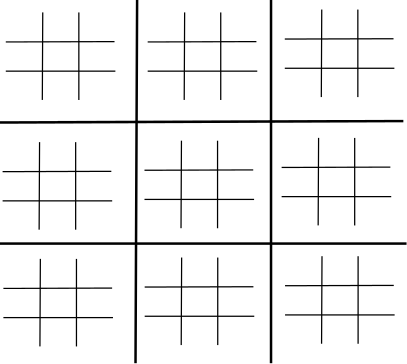}\end{center}
\caption{The board for a $3\times 3$ Super tic-tac-toe game}%
\label{3x3}%
\end{figure}

We refer to the larger $3\times 3$ grid as a board and the 9 smaller $3\times 3$ grids as fields. Player 1 plays X's and player 2 plays O's. Now, suppose that player 1 decides to put an X in the middle of the field in the upper left corner of the board. Then, player 2 must put an O anywhere in the middle field of the board. In other words, the placement in a field of a player determines which field on the board the next player must play (Figure \ref{game}). 

\begin{figure}[h!]
\begin{center}\includegraphics[scale=.33]{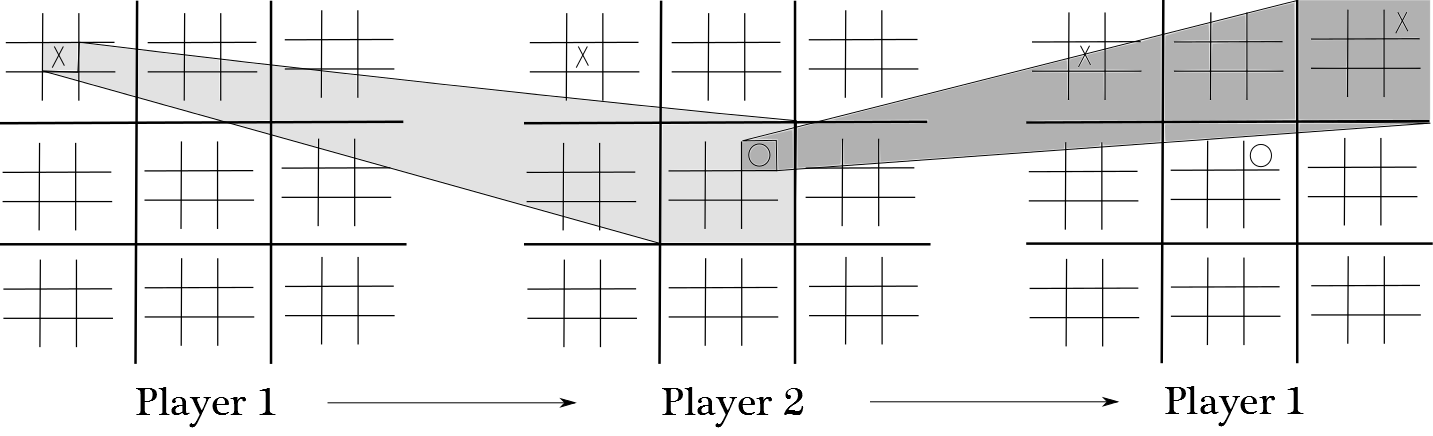}\end{center}
\caption{First 3 plays of a super tic-tac-toe game}%
\label{game}%
\end{figure}

 A game of super tic-tac-toe is won in a very similar fashion as a tradition tic-tac-toe game is won: You must get three X's or three O's in a row on the board. In order to get an X or and O on the board, the player  must win the game of tic-tac-toe in the corresponding field (Figure \ref{win}).

\begin{figure}[h!]
\begin{center}\includegraphics[scale=.41]{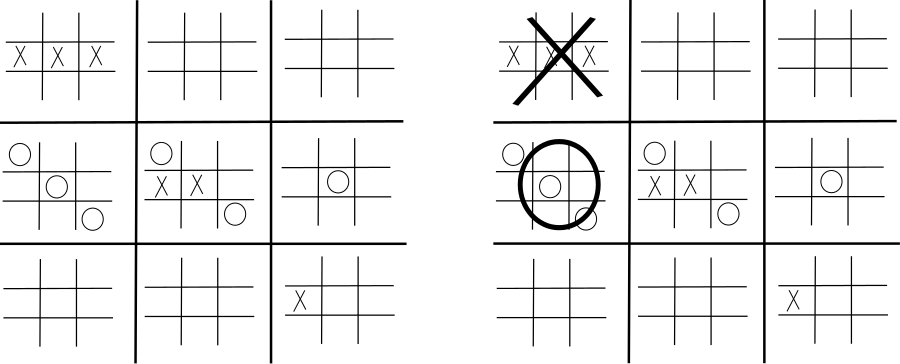}\end{center}
\caption{Winning a field in Super tic-tac-toe game}%
\label{win}%
\end{figure} 

It is possible for a game of tic-tac-toe to end in a draw; meaning that no player successfully placed three of their symbol in a row. If that happens in one of the fields on the board of super tic-tac-toe, then the player who has the most of their symbol in that field will win that square on the board. It is also possible that the current players move will send the next player to a field that has already been won. In this case, the next player is allowed to move anywhere on the board.

\section{Generalized Super Tic-Tac-Toe}

We can generalize super tic-tac-toe to be played on an a $n\times n$ tic-tac-toe board, where another $n\times n$ tic-tac-toe board is contained in each of its $n^2$ squares on the board. We refer to the larger $n\times n$ board as an $n$-board and the smaller $n\times n$ boards as $i$-fields.

\begin{definition} An $i$-field, $F_i$, is an $i\times i$ array with positions $a_{i,j}$ for $1\leq j\leq n^2$ where the label ordering for $j$ occur in a counter clockwise spiral. For example,

\begin{figure}[h]
\begin{center}
\begin{tikzpicture}[scale=.5]
\node at (-2,3.25) {$F_5=$};
\path[draw] (-.5,-.5)--(-.5,7);
\path[draw]  (7,-.5)--(7,7);
\path[draw] (-.5,-.5)--(7,-.5);
\path[draw]  (-.5,7)--(7,7);
\path[draw] (-.5,1)--(7,1);
\path[draw] (-.5,2.5)--(7,2.5);
\path[draw] (-.5,4)--(7,4);
\path[draw] (-.5,5.5)--(7,5.5);

\path[draw] (1,-.5)--(1,7);
\path[draw] (2.5,-.5)--(2.5,7);
\path[draw] (4,-.5)--(4,7);
\path[draw] (5.5,-.5)--(5.5,7);

\node at (0.25,6.25) {\tiny{$a_{5,1}$}};
\node at (0.25,4.75) {\tiny{$a_{5,2}$}};
\node at (0.25,3.25) {\tiny{$a_{5,3}$}};
\node at (0.25,1.75) {\tiny{$a_{5,4}$}};
\node at (0.25,.25) {\tiny{$a_{5,5}$}};

\node at (1.75,6.25) {\tiny{$a_{5,16}$}};
\node at (1.75,.25) {\tiny{$a_{5,6}$}};
\node at (3.25,6.25) {\tiny{$a_{5,15}$}};
\node at (3.25,.25) {\tiny{$a_{5,7}$}};
\node at (4.75,6.25) {\tiny{$a_{5,14}$}};
\node at (4.75,.25) {\tiny{$a_{5,8}$}};

\node at (6.25,6.25) {\tiny{$a_{5,13}$}};
\node at (6.25,4.75) {\tiny{$a_{5,12}$}};
\node at (6.25,3.25) {\tiny{$a_{5,11}$}};
\node at (6.25,1.75) {\tiny{$a_{5,10}$}};
\node at (6.25,.25) {\tiny{$a_{5,9}$}};
\node at (3.25,3.25) {\tiny{$a_{5,25}$}};
\node at (1.75,4.75) {\tiny{$a_{5,17}$}};
\node at (1.75,3.25) {\tiny{$a_{5,18}$}};
\node at (1.75,1.75) {\tiny{$a_{5,19}$}};

\node at (3.25,4.75) {\tiny{$a_{5,24}$}};
\node at (3.25,1.75) {\tiny{$a_{5,20}$}};

\node at (4.75,4.75) {\tiny{$a_{5,23}$}};
\node at (4.75,3.25) {\tiny{$a_{5,22}$}};
\node at (4.75,1.75) {\tiny{$a_{5,21}$}};
\end{tikzpicture}
\end{center}
\end{figure}

\vspace{.5cm}
\noindent Note the $a_{i,j}$ position is empty, an $O$, or an $X$.   
\end{definition}

\begin{definition}
An $n$-board, $B_n$,  is an $n\times n$ array of fields, $F_i$, written in a counter clockwise spiral. For example, 

\newpage
\begin{figure}[h]
\begin{center}
\begin{tikzpicture}[scale=.5]
\node at (-2,3.25) {$B_5=$};
\path[draw] (-.5,-.5)--(-.5,7);
\path[draw]  (7,-.5)--(7,7);
\path[draw] (-.5,-.5)--(7,-.5);
\path[draw]  (-.5,7)--(7,7);
\path[draw] (-.5,1)--(7,1);
\path[draw] (-.5,2.5)--(7,2.5);
\path[draw] (-.5,4)--(7,4);
\path[draw] (-.5,5.5)--(7,5.5);

\path[draw] (1,-.5)--(1,7);
\path[draw] (2.5,-.5)--(2.5,7);
\path[draw] (4,-.5)--(4,7);
\path[draw] (5.5,-.5)--(5.5,7);

\node at (0.25,6.25) {$F_1$};
\node at (0.25,4.75) {$F_2$};
\node at (0.25,3.25) {$F_3$};
\node at (0.25,1.75) {$F_4$};
\node at (0.25,.25) {$F_5$};

\node at (1.75,6.25) {$F_{16}$};
\node at (1.75,.25) {$F_6$};
\node at (3.25,6.25) {$F_{15}$};
\node at (3.25,.25) {$F_7$};
\node at (4.75,6.25) {$F_{14}$};
\node at (4.75,.25) {$F_8$};

\node at (6.25,6.25) {$F_{13}$};
\node at (6.25,4.75) {$F_{12}$};
\node at (6.25,3.25) {$F_{11}$};
\node at (6.25,1.75) {$F_{10}$};
\node at (6.25,.25) {$F_9$};
\node at (3.25,3.25) {$F_{25}$};
\node at (1.75,4.75) {$F_{17}$};
\node at (1.75,3.25) {$F_{18}$};
\node at (1.75,1.75) {$F_{19}$};

\node at (3.25,4.75) {$F_{24}$};
\node at (3.25,1.75) {$F_{20}$};

\node at (4.75,4.75) {$F_{23}$};
\node at (4.75,3.25) {$F_{22}$};
\node at (4.75,1.75) {$F_{21}$};
\end{tikzpicture}
\end{center}
\end{figure}

\vspace{.5cm}
 \noindent  The set of all $n$-boards is denoted by $\mathcal{B}_n$. We can represent an $n$-board $B_n\in\mathcal{B}_n$ by a list, $$[a_{1,1}, a_{1,2}, a_{1,3}, \dots, a_{1,n^2}, a_{2,1}, a_{2,2}, \dots a_{2,n^2}, \dots a_{n^2,1}, a_{n^2,2}, \dots a_{n^2, n^2}].$$

\noindent 

\end{definition}

\begin{figure}[ht]
\begin{center}
\begin{tikzpicture}

\path[draw] (0,-1)--(12,-1);
\path[draw] (0,-5)--(12,-5);

\path[draw] (4,-9)--(4,3);
\path[draw] (8,-9)--(8,3);

\path[draw] (.75,1.5)--(3.25,1.5);
\path[draw] (.75,.5)--(3.25,.5);
\path[draw] (1.5,2.25)--(1.5,-.25);
\path[draw] (2.5,2.25)--(2.5,-.25);
\node at (1,2) {$a_{1,1}$};
\node at (2,2) {$a_{1,8}$};
\node at (3,2) {$a_{1,7}$};
\node at (1,1) {$a_{1,2}$};
\node at (2,1) {$a_{1,9}$};
\node at (3,1) {$a_{1,6}$};
\node at (1,0) {$a_{1,3}$};
\node at (2,0) {$a_{1,4}$};
\node at (3,0) {$a_{1,5}$};

\path[draw] (4.75,1.5)--(7.25,1.5);
\path[draw] (4.75,.5)--(7.25,.5);
\path[draw] (5.5,2.25)--(5.5,-.25);
\path[draw] (6.5,2.25)--(6.5,-.25);
\node at (5,2) {$a_{8,1}$};
\node at (6,2) {$a_{8,8}$};
\node at (7,2) {$a_{8,7}$};
\node at (5,1) {$a_{8,2}$};
\node at (6,1) {$a_{8,9}$};
\node at (7,1) {$a_{8,6}$};
\node at (5,0) {$a_{8,3}$};
\node at (6,0) {$a_{8,4}$};
\node at (7,0) {$a_{8,5}$};

\path[draw] (8.75,1.5)--(11.25,1.5);
\path[draw] (8.75,.5)--(11.25,.5);
\path[draw] (9.5,2.25)--(9.5,-.25);
\path[draw] (10.5,2.25)--(10.5,-.25);
\node at (9,2) {$a_{7,1}$};
\node at (10,2) {$a_{7,8}$};
\node at (11,2) {$a_{7,7}$};
\node at (9,1) {$a_{7,2}$};
\node at (10,1) {$a_{7,9}$};
\node at (11,1) {$a_{7,6}$};
\node at (9,0) {$a_{7,3}$};
\node at (10,0) {$a_{7,4}$};
\node at (11,0) {$a_{7,5}$};

\path[draw] (.75,-2.5)--(3.25,-2.5);
\path[draw] (.75,-3.5)--(3.25,-3.5);
\path[draw] (1.5,-1.75)--(1.5,-4.25);
\path[draw] (2.5,-1.75)--(2.5,-4.25);
\node at (1,-2) {$a_{2,1}$};
\node at (2,-2) {$a_{2,8}$};
\node at (3,-2) {$a_{2,7}$};
\node at (1,-3) {$a_{2,2}$};
\node at (2,-3) {$a_{2,9}$};
\node at (3,-3) {$a_{2,6}$};
\node at (1,-4) {$a_{2,3}$};
\node at (2,-4) {$a_{2,4}$};
\node at (3,-4) {$a_{2,5}$};

\path[draw] (4.75,-2.5)--(7.25,-2.5);
\path[draw] (4.75,-3.5)--(7.25,-3.5);
\path[draw] (5.5,-1.75)--(5.5,-4.25);
\path[draw] (6.5,-1.75)--(6.5,-4.25);
\node at (5,-2) {$a_{9,1}$};
\node at (6,-2) {$a_{9,8}$};
\node at (7,-2) {$a_{9,7}$};
\node at (5,-3) {$a_{9,2}$};
\node at (6,-3) {$a_{9,9}$};
\node at (7,-3) {$a_{9,6}$};
\node at (5,-4) {$a_{9,3}$};
\node at (6,-4) {$a_{9,4}$};
\node at (7,-4) {$a_{9,5}$};

\path[draw] (8.75,-2.5)--(11.25,-2.5);
\path[draw] (8.75,-3.5)--(11.25,-3.5);
\path[draw] (9.5,-1.75)--(9.5,-4.25);
\path[draw] (10.5,-1.75)--(10.5,-4.25);
\node at (9,-2) {$a_{6,1}$};
\node at (10,-2) {$a_{6,8}$};
\node at (11,-2) {$a_{6,7}$};
\node at (9,-3) {$a_{6,2}$};
\node at (10,-3) {$a_{6,9}$};
\node at (11,-3) {$a_{6,6}$};
\node at (9,-4) {$a_{6,3}$};
\node at (10,-4) {$a_{6,4}$};
\node at (11,-4) {$a_{6,5}$};

\path[draw] (.75,-6.5)--(3.25,-6.5);
\path[draw] (.75,-7.5)--(3.25,-7.5);
\path[draw] (1.5,-5.75)--(1.5,-8.25);
\path[draw] (2.5,-5.75)--(2.5,-8.25);
\node at (1,-6) {$a_{3,1}$};
\node at (2,-6) {$a_{3,8}$};
\node at (3,-6) {$a_{3,7}$};
\node at (1,-7) {$a_{3,2}$};
\node at (2,-7) {$a_{3,9}$};
\node at (3,-7) {$a_{3,6}$};
\node at (1,-8) {$a_{3,3}$};
\node at (2,-8) {$a_{3,4}$};
\node at (3,-8) {$a_{3,5}$};

\path[draw] (4.75,-6.5)--(7.25,-6.5);
\path[draw] (4.75,-7.5)--(7.25,-7.5);
\path[draw] (5.5,-5.75)--(5.5,-8.25);
\path[draw] (6.5,-5.75)--(6.5,-8.25);
\node at (5,-6) {$a_{4,1}$};
\node at (6,-6) {$a_{4,8}$};
\node at (7,-6) {$a_{4,7}$};
\node at (5,-7) {$a_{4,2}$};
\node at (6,-7) {$a_{4,9}$};
\node at (7,-7) {$a_{4,6}$};
\node at (5,-8) {$a_{4,3}$};
\node at (6,-8) {$a_{4,4}$};
\node at (7,-8) {$a_{4,5}$};

\path[draw] (8.75,-6.5)--(11.25,-6.5);
\path[draw] (8.75,-7.5)--(11.25,-7.5);
\path[draw] (9.5,-5.75)--(9.5,-8.25);
\path[draw] (10.5,-5.75)--(10.5,-8.25);
\node at (9,-6) {$a_{5,1}$};
\node at (10,-6) {$a_{5,8}$};
\node at (11,-6) {$a_{5,7}$};
\node at (9,-7) {$a_{5,2}$};
\node at (10,-7) {$a_{5,9}$};
\node at (11,-7) {$a_{5,6}$};
\node at (9,-8) {$a_{5,3}$};
\node at (10,-8) {$a_{5,4}$};
\node at (11,-8) {$a_{5,5}$};
\end{tikzpicture}
\end{center}
\caption{An example of a $3\times 3$ super tic-tac-toe board, which we call $B_3$}
\end{figure}

To win a square on the $n$-board, the player must place $n$ of their symbols horizontally, vertically, or diagonally in a row in the corresponding $i$-field. In the case of a tie on an $i$-field, the player that has more X's or O's on the smaller board wins the larger board. 

We will follow the rules laid out in \cite{2by2}. While the game is played, the current move a player makes in an $i$-field dictates the square on the $n$-board where the next player must move. In particular, if a player moves in position $a_{i,j}$, then the next player must move in $F_j$. For example, if the current player moves in position $a_{8,4}$ then the next player must move in $F_4$. If a player is sent to a square on the $n$-board that is already won, the player can choose any $i$-field to play in.

We will focus our work on impartial super tic-tac-toe with misere rules. Both players play X and the player to complete $n$ in a row on the board loses. In \cite{notakto}, the structure of misere tic-tac-toe was studied for disjunctive $3\times 3$ games, that is a game with an arbitrary number of $3 \times 3$ boards. 

We want to study the structure of the board in super tic-tac-toe. In particular, we want to impose a set of group actions and look at the symmetries in different boards. We will look at the structure of any size $n\times n$ super tic-tac-toe board and prove that the group actions are isomorphic to a dihedral group $D_{m}$.

\section{Decomposition of a $n\times n$ Square}
Given an $n\times n$ square board, we number each square starting with 1 in the upper left corner and moving counter clockwise. If the next square has already been filled with a number, place the next number in the square below the current position and repeat until all $n^2$ squares have been filled (Figure \ref{5x5square}). We refer to such a square as a {\it{numbered square}}.

\begin{figure}[h]
\begin{center}
\begin{tikzpicture}[scale=.5]
\path[draw] (-.5,1)--(7,1);
\path[draw] (-.5,2.5)--(7,2.5);
\path[draw] (-.5,4)--(7,4);
\path[draw] (-.5,5.5)--(7,5.5);

\path[draw] (1,0)--(1,6.75);
\path[draw] (2.5,0)--(2.5,6.75);
\path[draw] (4,0)--(4,6.75);
\path[draw] (5.5,0)--(5.5,6.75);

\node at (0.25,6.25) {1};
\node at (0.25,4.75) {2};
\node at (0.25,3.25) {3};
\node at (0.25,1.75) {4};
\node at (0.25,.25) {5};

\node at (1.75,6.25) {16};
\node at (1.75,4.75) {17};
\node at (1.75,3.25) {18};
\node at (1.75,1.75) {19};
\node at (1.75,.25) {6};

\node at (3.25,6.25) {15};
\node at (3.25,4.75) {24};
\node at (3.25,3.25) {25};
\node at (3.25,1.75) {20};
\node at (3.25,.25) {7};

\node at (4.75,6.25) {14};
\node at (4.75,4.75) {23};
\node at (4.75,3.25) {22};
\node at (4.75,1.75) {21};
\node at (4.75,.25) {8};

\node at (6.25,6.25) {13};
\node at (6.25,4.75) {12};
\node at (6.25,3.25) {11};
\node at (6.25,1.75) {10};
\node at (6.25,.25) {9};
\end{tikzpicture}
\end{center}
\caption{A $5\times 5$ Numbered Square}\label{5x5square}
\end{figure}

Given an $n\times n$ numbered square, we can decompose the board into layers (Figure 6). If $n$ is odd, layer 1 is the inner most $1\times 1$ square, layer 2 is the set of squares surrounding layer 1, layer 3 is the outer set of squares surrounding layer 2, and so forth.  If $n$ is even, layer 1 is the inner most $2\times 2$ square, layer 2 is the set of square surrounding layer 1, and so forth. For an $n\times n$ square, there are $\displaystyle{\left\lfloor\dfrac{n+1}{2}\right\rfloor}$ layers.

\begin{definition}
Given an $n\times n$ numbered square, the {\bf{$k^{th}$ level set}} is the list of numbers $[k_0, k_1, k_2,....,k_j]$ in layer k where $k_i\leq k_{i+1}$. Note that, $k_i=k_{i-(j+1)}$.
\end{definition}

If $n$ is odd, then the $1^{st}$ level set has cardinality 1, and the $k^{th}$ level set has cardinality $8(k-1)$. If $n$ is even, then the $k^{th}$ level set has cardinality $4+8(k-1)$. 

\begin{example}Consider the $5\times 5$ numbered square in Figure \ref{5x5square}. There are $\displaystyle{\left\lfloor\dfrac{5+1}{2}\right\rfloor=3}$ layers. The $1^{st}$ layer set is $[25]$, the $2^{nd}$ layer set is $[17,18,19,20,21,22,23,24]$, and the $3^{rd}$ layer set is $[1,2,3,4,5,6,7,8,9,10,11,12,13,14,15,16]$ as seen in Figure 6.

\begin{figure}[h]
\begin{center}
\begin{tikzpicture}[scale=.5]
\path[draw] (-.5,-.5)--(-.5,7);
\path[draw]  (7,-.5)--(7,7);
\path[draw] (-.5,-.5)--(7,-.5);
\path[draw]  (-.5,7)--(7,7);
\path[draw] (-.5,1)--(7,1);
\path[draw] (-.5,2.5)--(7,2.5);
\path[draw] (-.5,4)--(7,4);
\path[draw] (-.5,5.5)--(7,5.5);

\path[draw] (1,-.5)--(1,7);
\path[draw] (2.5,-.5)--(2.5,7);
\path[draw] (4,-.5)--(4,7);
\path[draw] (5.5,-.5)--(5.5,7);

\node at (0.25,6.25) {1};
\node at (0.25,4.75) {2};
\node at (0.25,3.25) {3};
\node at (0.25,1.75) {4};
\node at (0.25,.25) {5};

\node at (1.75,6.25) {16};
\node at (1.75,.25) {6};
\node at (3.25,6.25) {15};
\node at (3.25,.25) {7};
\node at (4.75,6.25) {14};
\node at (4.75,.25) {8};

\node at (6.25,6.25) {13};
\node at (6.25,4.75) {12};
\node at (6.25,3.25) {11};
\node at (6.25,1.75) {10};
\node at (6.25,.25) {9};

\path[draw] (11,1)--(15.5,1);
\path[draw] (11,2.5)--(15.5,2.5);
\path[draw] (11,4)--(15.5,4);
\path[draw] (11,5.5)--(15.5,5.5);

\path[draw] (11,1)--(11,5.5);
\path[draw] (12.5,1)--(12.5,5.5);
\path[draw] (14,1)--(14,5.5);
\path[draw] (15.5,1)--(15.5,5.5);

\node at (11.75,4.75) {17};
\node at (11.75,3.25) {18};
\node at (11.75,1.75) {19};

\node at (13.25,4.75) {24};
\node at (13.25,1.75) {20};

\node at (14.75,4.75) {23};
\node at (14.75,3.25) {22};
\node at (14.75,1.75) {21};

\path[draw] (19.5,2.5)--(21,2.5);
\path[draw] (19.5,4)--(21,4);
\path[draw] (19.5,2.5)--(19.5,4);
\path[draw] (21,2.5)--(21,4);

\node at (20.25,3.25) {25};
\end{tikzpicture}
\end{center}

\caption{From Left to Right: Layer 3, Layer 2, and Layer 1 of a $5\times5$ square}\label{layers}
\end{figure}
\end{example}

\section{Rotations and Reflections of a Super Tic-Tac-Toe Board}

\subsection{Rotations of a Super Tic-Tac-Toe Board}
Consider an $n\times n$ numbered square with $[k_0, k_1, k_2,\dots, k_j]$ as the $k^{th}$ level set. A {\it{rotation of layer $k$}} is a counterclockwise rotation represented by the permutation $\sigma_k=(k_0, k_1, k_2,\dots, k_j)\in Sym(S)$, where $S=\{k_0, k_1,\dots,k_{j}\}$.\\

\noindent A {\it{rotation of an $n\times n$ numbered square}} is a counter clockwise rotation in each layer given by $$\sigma=\sigma_1\sigma_2\sigma_3\cdots\sigma_l,$$ where $l=\displaystyle{\left\lfloor\dfrac{n+1}{2}\right\rfloor}$ is the number of layers in an $n\times n$ square (Figure \ref{rot_square}). 

\begin{figure}[h]
\begin{center}
\begin{tikzpicture}[scale=.5]
\path[draw] (-.5,-.5)--(-.5,7);
\path[draw]  (7,-.5)--(7,7);
\path[draw] (-.5,-.5)--(7,-.5);
\path[draw]  (-.5,7)--(7,7);
\path[draw] (-.5,1)--(7,1);
\path[draw] (-.5,2.5)--(7,2.5);
\path[draw] (-.5,4)--(7,4);
\path[draw] (-.5,5.5)--(7,5.5);

\path[draw] (1,-.5)--(1,7);
\path[draw] (2.5,-.5)--(2.5,7);
\path[draw] (4,-.5)--(4,7);
\path[draw] (5.5,-.5)--(5.5,7);

\node at (0.25,6.25) {1};
\node at (0.25,4.75) {2};
\node at (0.25,3.25) {3};
\node at (0.25,1.75) {4};
\node at (0.25,.25) {5};

\node at (1.75,6.25) {16};
\node at (1.75,.25) {6};
\node at (3.25,6.25) {15};
\node at (3.25,.25) {7};
\node at (4.75,6.25) {14};
\node at (4.75,.25) {8};

\node at (6.25,6.25) {13};
\node at (6.25,4.75) {12};
\node at (6.25,3.25) {11};
\node at (6.25,1.75) {10};
\node at (6.25,.25) {9};
\node at (3.25,3.25) {25};
\node at (1.75,4.75) {17};
\node at (1.75,3.25) {18};
\node at (1.75,1.75) {19};

\node at (3.25,4.75) {24};
\node at (3.25,1.75) {20};

\node at (4.75,4.75) {23};
\node at (4.75,3.25) {22};
\node at (4.75,1.75) {21};

\draw[->] (10,3.25)--(12,3.25);
\node at (11,4.25) {$\sigma$};

\path[draw] (14.5,-.5)--(14.5,7);
\path[draw]  (22,-.5)--(22,7);
\path[draw] (14.5,-.5)--(22,-.5);
\path[draw]  (14.5,7)--(22,7);
\path[draw] (14.5,1)--(22,1);
\path[draw] (14.5,2.5)--(22,2.5);
\path[draw] (14.5,4)--(22,4);
\path[draw] (14.5,5.5)--(22,5.5);

\path[draw] (16,-.5)--(16,7);
\path[draw] (17.5,-.5)--(17.5,7);
\path[draw] (19,-.5)--(19,7);
\path[draw] (20.5,-.5)--(20.5,7);

\node at (15.25,6.25) {16};
\node at (15.25,4.75) {1};
\node at (15.25,3.25) {2};
\node at (15.25,1.75) {3};
\node at (15.25,.25) {4};

\node at (16.75,6.25) {15};
\node at (16.75,.25) {5};
\node at (18.25,6.25) {14};
\node at (18.25,.25) {6};
\node at (19.75,6.25) {13};
\node at (19.75,.25) {7};

\node at (21.25,6.25) {12};
\node at (21.25,4.75) {11};
\node at (21.25,3.25) {10};
\node at (21.25,1.75) {9};
\node at (21.25,.25) {8};
\node at (18.25,3.25) {25};
\node at (16.75,4.75) {24};
\node at (16.75,3.25) {17};
\node at (16.75,1.75) {18};

\node at (18.25,4.75) {23};
\node at (18.25,1.75) {19};

\node at (19.75,4.75) {22};
\node at (19.75,3.25) {21};
\node at (19.75,1.75) {20};
\end{tikzpicture}
\end{center}
\caption{Rotation of an $5\time 5$ numbered square}\label{rot_square}
\end{figure}


Let $[a_{1,1},\dots, a_{1,n^2}, a_{2,1}, \dots a_{2,n^2}, \dots a_{n^2,1}, \dots a_{n^2, n^2}]$ represent $B_n\in\mathcal{B}_n$. We define $\Sigma: \mathcal{B}_n\rightarrow\mathcal{B}_n$ by  \begin{align*}\Sigma([a_{1,1}, \dots, a_{1,n^2}, a_{2,1}, \dots a_{2,n^2}, \dots &a_{n^2,1}, \dots a_{n^2, n^2}])=\\&[a_{\sigma(1),\sigma(1)}, a_{\sigma(1),\sigma(2)}, \dots, a_{\sigma(n^2),\sigma(n^2)}].\end{align*}

\noindent In other words, $\Sigma$ is a counterclockwise rotation of an $n$-board whose fields have also been rotated in a counterclockwise manner (Figure \ref{rot_board}).

\begin{figure}[h]
\begin{center}
\begin{tikzpicture}[scale=.5]
\path[draw] (0,0)--(6,0);
\path[draw] (3,-3)--(3,3);
\path[draw] (1.35,-2.35)--(1.35,-.75);
\path[draw] (1.35,2.35)--(1.35,.75);
\path[draw] (4.65,-2.35)--(4.65,-.75);
\path[draw] (4.65,2.35)--(4.65,.75);

\path[draw] (.15,1.5)--(2.35,1.5);
\path[draw] (3.65,1.5)--(5.85,1.5);
\path[draw] (.15,-1.5)--(2.35,-1.5);
\path[draw] (3.65,-1.5)--(5.85,-1.5);

\node at (0.5,2) {$a_{1,1}$};
\node at (2.25,2) {$a_{1,4}$};
\node at (0.5,.75) {$a_{1,2}$};
\node at (2.25,.75) {$a_{1,3}$};

\node at (0.5,-2.25) {$a_{2,2}$};
\node at (2.25,-2.25) {$a_{2,3}$};
\node at (0.5,-1) {$a_{2,1}$};
\node at (2.25,-1) {$a_{2,4}$};

\node at (3.85,2) {$a_{4,1}$};
\node at (5.5,2) {$a_{4,4}$};
\node at (3.85,.75) {$a_{4,2}$};
\node at (5.5,.75) {$a_{4,3}$};

\node at (3.85,-2.25) {$a_{3,2}$};
\node at (5.5,-2.25) {$a_{3,3}$};
\node at (3.85,-1) {$a_{3,1}$};
\node at (5.5,-1) {$a_{3,4}$};


\draw[->] (7,0)--(9,0);
\node at (8,1) {$\Sigma$};

\path[draw] (10,0)--(16,0);
\path[draw] (13,-3)--(13,3);
\path[draw] (11.35,-2.35)--(11.35,-.75);
\path[draw] (11.35,2.35)--(11.35,.75);
\path[draw] (14.65,-2.35)--(14.65,-.75);
\path[draw] (14.65,2.35)--(14.65,.75);

\path[draw] (10.15,1.5)--(12.35,1.5);
\path[draw] (13.65,1.5)--(15.85,1.5);
\path[draw] (10.15,-1.5)--(12.35,-1.5);
\path[draw] (13.65,-1.5)--(15.85,-1.5);

\node at (10.5,2) {$a_{4,4}$};
\node at (12.25,2) {$a_{4,3}$};
\node at (10.5,.75) {$a_{4,1}$};
\node at (12.25,.75) {$a_{4,2}$};

\node at (10.5,-2.25) {$a_{1,1}$};
\node at (12.25,-2.25) {$a_{1,2}$};
\node at (10.5,-1) {$a_{1,4}$};
\node at (12.25,-1) {$a_{1,3}$};

\node at (13.85,2) {$a_{3,4}$};
\node at (15.5,2) {$a_{3,3}$};
\node at (13.85,.75) {$a_{3,1}$};
\node at (15.5,.75) {$a_{3,2}$};

\node at (13.85,-2.25) {$a_{2,1}$};
\node at (15.5,-2.25) {$a_{2,2}$};
\node at (13.85,-1) {$a_{2,4}$};
\node at (15.5,-1) {$a_{2,3}$};
\end{tikzpicture}
\end{center}
\caption{Rotation of a $2\times 2$ super tic-tac-toe board}\label{rot_board}
\end{figure}

\subsection{Reflections of a Super Tic-Tac-Toe Board}

Consider an $n\times n$ numbered square with $l$ levels, with $[k_0, k_1, k_2,\dots, k_j]$ as the $k^{th}$ level set. A {\it{reflection of layer $k$}} is a reflection about the axis from the upper left corner of layer $k$ to the lower right corner of layer $k$ defined by $$\rho_k=(k_1,k_{-1})(k_2,k_{-2})(k_3,k_{-3}).....(k_{(j-1)/2}, k_{-(j-1)/2}).$$

\noindent A {\it{reflection}} of an $ n \times n$ square is a reflection about the axis from the upper left corner of the square to the lower right corner of the square is defined by $$\rho=\rho_1\rho_2\rho_3\cdots\rho_l,$$ where where $l=\displaystyle{\left\lfloor\dfrac{n+1}{2}\right\rfloor}$ is the number of layers in an $n\times n$ square (Figure \ref{ref_square}).

\begin{figure}[h]
\begin{center}
\begin{tikzpicture}[scale=.5]
\path[draw] (-.5,-.5)--(-.5,7);
\path[draw]  (7,-.5)--(7,7);
\path[draw] (-.5,-.5)--(7,-.5);
\path[draw]  (-.5,7)--(7,7);
\path[draw] (-.5,1)--(7,1);
\path[draw] (-.5,2.5)--(7,2.5);
\path[draw] (-.5,4)--(7,4);
\path[draw] (-.5,5.5)--(7,5.5);

\path[draw] (1,-.5)--(1,7);
\path[draw] (2.5,-.5)--(2.5,7);
\path[draw] (4,-.5)--(4,7);
\path[draw] (5.5,-.5)--(5.5,7);

\node at (0.25,6.25) {1};
\node at (0.25,4.75) {2};
\node at (0.25,3.25) {3};
\node at (0.25,1.75) {4};
\node at (0.25,.25) {5};

\node at (1.75,6.25) {16};
\node at (1.75,.25) {6};
\node at (3.25,6.25) {15};
\node at (3.25,.25) {7};
\node at (4.75,6.25) {14};
\node at (4.75,.25) {8};

\node at (6.25,6.25) {13};
\node at (6.25,4.75) {12};
\node at (6.25,3.25) {11};
\node at (6.25,1.75) {10};
\node at (6.25,.25) {9};
\node at (3.25,3.25) {25};
\node at (1.75,4.75) {17};
\node at (1.75,3.25) {18};
\node at (1.75,1.75) {19};

\node at (3.25,4.75) {24};
\node at (3.25,1.75) {20};

\node at (4.75,4.75) {23};
\node at (4.75,3.25) {22};
\node at (4.75,1.75) {21};

\draw[->] (10,3.25)--(12,3.25);
\node at (11,4.25) {$\rho$};

\path[draw] (14.5,-.5)--(14.5,7);
\path[draw]  (22,-.5)--(22,7);
\path[draw] (14.5,-.5)--(22,-.5);
\path[draw]  (14.5,7)--(22,7);
\path[draw] (14.5,1)--(22,1);
\path[draw] (14.5,2.5)--(22,2.5);
\path[draw] (14.5,4)--(22,4);
\path[draw] (14.5,5.5)--(22,5.5);

\path[draw] (16,-.5)--(16,7);
\path[draw] (17.5,-.5)--(17.5,7);
\path[draw] (19,-.5)--(19,7);
\path[draw] (20.5,-.5)--(20.5,7);

\node at (15.25,6.25) {1};
\node at (15.25,4.75) {16};
\node at (15.25,3.25) {15};
\node at (15.25,1.75) {14};
\node at (15.25,.25) {13};

\node at (16.75,6.25) {2};
\node at (16.75,.25) {12};
\node at (18.25,6.25) {3};
\node at (18.25,.25) {11};
\node at (19.75,6.25) {4};
\node at (19.75,.25) {10};

\node at (21.25,6.25) {5};
\node at (21.25,4.75) {6};
\node at (21.25,3.25) {7};
\node at (21.25,1.75) {8};
\node at (21.25,.25) {9};
\node at (18.25,3.25) {25};
\node at (16.75,4.75) {17};
\node at (16.75,3.25) {24};
\node at (16.75,1.75) {23};

\node at (18.25,4.75) {18};
\node at (18.25,1.75) {22};

\node at (19.75,4.75) {19};
\node at (19.75,3.25) {20};
\node at (19.75,1.75) {21};
\end{tikzpicture}
\end{center}
\caption{Reflection of an $5\time 5$ numbered square}\label{ref_square}
\end{figure}

Let $[a_{1,1},\dots, a_{1,n^2}, a_{2,1}, \dots a_{2,n^2}, \dots a_{n^2,1}, \dots a_{n^2, n^2}]$ represent $B_n\in\mathcal{B}_n$. We define $P: \mathcal{B}_n\rightarrow\mathcal{B}_n$ by \begin{align*}P([a_{1,1}, \dots, a_{1,n^2}, a_{2,1}, \dots a_{2,n^2}, \dots &a_{n^2,1}, \dots a_{n^2, n^2}])=\\&[a_{\rho(1),\rho(1)}, a_{\rho(1),\rho(2)}, \dots, a_{\rho(n^2),\rho(n^2)}].\end{align*}

\noindent In other words, $P$ is a reflection of an $n$-board whose fields have also been reflected (Figure \ref{ref_board}). 
\begin{figure}[h]
\begin{center}
\begin{tikzpicture}[scale=.5]
\path[draw] (0,0)--(6,0);
\path[draw] (3,-3)--(3,3);
\path[draw] (1.35,-2.35)--(1.35,-.75);
\path[draw] (1.35,2.35)--(1.35,.75);
\path[draw] (4.65,-2.35)--(4.65,-.75);
\path[draw] (4.65,2.35)--(4.65,.75);

\path[draw] (.15,1.5)--(2.35,1.5);
\path[draw] (3.65,1.5)--(5.85,1.5);
\path[draw] (.15,-1.5)--(2.35,-1.5);
\path[draw] (3.65,-1.5)--(5.85,-1.5);

\node at (0.5,2) {$a_{1,1}$};
\node at (2.25,2) {$a_{1,4}$};
\node at (0.5,.75) {$a_{1,2}$};
\node at (2.25,.75) {$a_{1,3}$};

\node at (0.5,-2.25) {$a_{2,2}$};
\node at (2.25,-2.25) {$a_{2,3}$};
\node at (0.5,-1) {$a_{2,1}$};
\node at (2.25,-1) {$a_{2,4}$};

\node at (3.85,2) {$a_{4,1}$};
\node at (5.5,2) {$a_{4,4}$};
\node at (3.85,.75) {$a_{4,2}$};
\node at (5.5,.75) {$a_{4,3}$};

\node at (3.85,-2.25) {$a_{3,2}$};
\node at (5.5,-2.25) {$a_{3,3}$};
\node at (3.85,-1) {$a_{3,1}$};
\node at (5.5,-1) {$a_{3,4}$};

\draw[->] (7,0)--(9,0);
\node at (8,1) {$P$};

\path[draw] (10,0)--(16,0);
\path[draw] (13,-3)--(13,3);
\path[draw] (11.35,-2.35)--(11.35,-.75);
\path[draw] (11.35,2.35)--(11.35,.75);
\path[draw] (14.65,-2.35)--(14.65,-.75);
\path[draw] (14.65,2.35)--(14.65,.75);

\path[draw] (10.15,1.5)--(12.35,1.5);
\path[draw] (13.65,1.5)--(15.85,1.5);
\path[draw] (10.15,-1.5)--(12.35,-1.5);
\path[draw] (13.65,-1.5)--(15.85,-1.5);

\node at (10.5,2) {$a_{1,1}$};
\node at (12.25,2) {$a_{1,2}$};
\node at (10.5,.75) {$a_{1,4}$};
\node at (12.25,.75) {$a_{1,3}$};

\node at (10.5,-2.25) {$a_{4,4}$};
\node at (12.25,-2.25) {$a_{4,3}$};
\node at (10.5,-1) {$a_{4,1}$};
\node at (12.25,-1) {$a_{4,2}$};

\node at (13.85,2) {$a_{2,1}$};
\node at (15.5,2) {$a_{2,2}$};
\node at (13.85,.75) {$a_{2,4}$};
\node at (15.5,.75) {$a_{2,3}$};

\node at (13.85,-2.25) {$a_{3,4}$};
\node at (15.5,-2.25) {$a_{3,3}$};
\node at (13.85,-1) {$a_{3,1}$};
\node at (15.5,-1) {$a_{3,2}$};
\end{tikzpicture}
\end{center}
\caption{Reflection of a $2\times 2$ super tic-tac-toe board}\label{ref_board}
\end{figure}

\subsection{Group Action of a Super Tic-Tac-Toe Board}
Before we state our main result, we present an example to help motivate the results. 

\begin{example} Consider the $2\times 2$ super tic-tac-toe board. 
If we continue to rotate the boards, we see that there are 4 game board configurations (Figure \ref{rot_board_mult}). If we reflect the board, we see that there are 2 game board configurations (Figure \ref{ref_board_mult}).


\begin{figure}[h]
\begin{center}
\begin{tikzpicture}[scale=.45]
\path[draw] (0,0)--(6,0);
\path[draw] (3,-3)--(3,3);
\path[draw] (1.35,-2.35)--(1.35,-.75);
\path[draw] (1.35,2.35)--(1.35,.75);
\path[draw] (4.65,-2.35)--(4.65,-.75);
\path[draw] (4.65,2.35)--(4.65,.75);

\path[draw] (.15,1.5)--(2.35,1.5);
\path[draw] (3.65,1.5)--(5.85,1.5);
\path[draw] (.15,-1.5)--(2.35,-1.5);
\path[draw] (3.65,-1.5)--(5.85,-1.5);

\node at (0.5,2) {$a_{1,1}$};
\node at (2.25,2) {$a_{1,4}$};
\node at (0.5,.75) {$a_{1,2}$};
\node at (2.25,.75) {$a_{1,3}$};

\node at (0.5,-2.25) {$a_{2,2}$};
\node at (2.25,-2.25) {$a_{2,3}$};
\node at (0.5,-1) {$a_{2,1}$};
\node at (2.25,-1) {$a_{2,4}$};

\node at (3.85,2) {$a_{4,1}$};
\node at (5.5,2) {$a_{4,4}$};
\node at (3.85,.75) {$a_{4,2}$};
\node at (5.5,.75) {$a_{4,3}$};

\node at (3.85,-2.25) {$a_{3,2}$};
\node at (5.5,-2.25) {$a_{3,3}$};
\node at (3.85,-1) {$a_{3,1}$};
\node at (5.5,-1) {$a_{3,4}$};


\draw[->] (7,0)--(9,0);
\node at (8,1) {$\Sigma$};

\path[draw] (10,0)--(16,0);
\path[draw] (13,-3)--(13,3);
\path[draw] (11.35,-2.35)--(11.35,-.75);
\path[draw] (11.35,2.35)--(11.35,.75);
\path[draw] (14.65,-2.35)--(14.65,-.75);
\path[draw] (14.65,2.35)--(14.65,.75);

\path[draw] (10.15,1.5)--(12.35,1.5);
\path[draw] (13.65,1.5)--(15.85,1.5);
\path[draw] (10.15,-1.5)--(12.35,-1.5);
\path[draw] (13.65,-1.5)--(15.85,-1.5);

\node at (10.5,2) {$a_{4,4}$};
\node at (12.25,2) {$a_{4,3}$};
\node at (10.5,.75) {$a_{4,1}$};
\node at (12.25,.75) {$a_{4,2}$};

\node at (10.5,-2.25) {$a_{1,1}$};
\node at (12.25,-2.25) {$a_{1,2}$};
\node at (10.5,-1) {$a_{1,4}$};
\node at (12.25,-1) {$a_{1,3}$};

\node at (13.85,2) {$a_{3,4}$};
\node at (15.5,2) {$a_{3,3}$};
\node at (13.85,.75) {$a_{3,1}$};
\node at (15.5,.75) {$a_{3,2}$};

\node at (13.85,-2.25) {$a_{2,1}$};
\node at (15.5,-2.25) {$a_{2,2}$};
\node at (13.85,-1) {$a_{2,4}$};
\node at (15.5,-1) {$a_{2,3}$};

\draw[->] (17,0)--(19,0);
\node at (18,1) {$\Sigma$};

\path[draw] (20,0)--(26,0);
\path[draw] (23,-3)--(23,3);
\path[draw] (21.35,-2.35)--(21.35,-.75);
\path[draw] (21.35,2.35)--(21.35,.75);
\path[draw] (24.65,-2.35)--(24.65,-.75);
\path[draw] (24.65,2.35)--(24.65,.75);

\path[draw] (20.15,1.5)--(22.35,1.5);
\path[draw] (23.65,1.5)--(25.85,1.5);
\path[draw] (20.15,-1.5)--(22.35,-1.5);
\path[draw] (23.65,-1.5)--(25.85,-1.5);

\node at (20.5,2) {$a_{3,3}$};
\node at (22.25,2) {$a_{3,2}$};
\node at (20.5,.75) {$a_{3,4}$};
\node at (22.25,.75) {$a_{3,1}$};

\node at (20.5,-2.25) {$a_{4,4}$};
\node at (22.25,-2.25) {$a_{4,1}$};
\node at (20.5,-1) {$a_{4,3}$};
\node at (22.25,-1) {$a_{4,2}$};

\node at (23.85,2) {$a_{2,3}$};
\node at (25.5,2) {$a_{2,2}$};
\node at (23.85,.75) {$a_{2,4}$};
\node at (25.5,.75) {$a_{2,1}$};

\node at (23.85,-2.25) {$a_{1,4}$};
\node at (25.5,-2.25) {$a_{1,1}$};
\node at (23.85,-1) {$a_{1,3}$};
\node at (25.5,-1) {$a_{1,2}$};
\draw[->] (27,0)--(29,0);
\node at (28,1) {$\Sigma$};

\path[draw] (30,0)--(36,0);
\path[draw] (33,-3)--(33,3);
\path[draw] (31.35,-2.35)--(31.35,-.75);
\path[draw] (31.35,2.35)--(31.35,.75);
\path[draw] (34.65,-2.35)--(34.65,-.75);
\path[draw] (34.65,2.35)--(34.65,.75);

\path[draw] (30.15,1.5)--(32.35,1.5);
\path[draw] (33.65,1.5)--(35.85,1.5);
\path[draw] (30.15,-1.5)--(32.35,-1.5);
\path[draw] (33.65,-1.5)--(35.85,-1.5);

\node at (30.5,2) {$a_{2,2}$};
\node at (32.25,2) {$a_{2,1}$};
\node at (30.5,.75) {$a_{2,3}$};
\node at (32.25,.75) {$a_{2,4}$};

\node at (30.5,-2.25) {$a_{3,3}$};
\node at (32.25,-2.25) {$a_{3,4}$};
\node at (30.5,-1) {$a_{3,2}$};
\node at (32.25,-1) {$a_{3,1}$};

\node at (33.85,2) {$a_{1,2}$};
\node at (35.5,2) {$a_{1,1}$};
\node at (33.85,.75) {$a_{1,3}$};
\node at (35.5,.75) {$a_{1,4}$};

\node at (33.85,-2.25) {$a_{4,3}$};
\node at (35.5,-2.25) {$a_{4,4}$};
\node at (33.85,-1) {$a_{4,2}$};
\node at (35.5,-1) {$a_{4,1}$};
\end{tikzpicture}
\end{center}
\caption{Rotations of a $2\times 2$ super tic-tac-toe board}\label{rot_board_mult}
\end{figure}

\begin{figure}[h]
\begin{center}
\begin{tikzpicture}[scale=.5]
\path[draw] (0,0)--(6,0);
\path[draw] (3,-3)--(3,3);
\path[draw] (1.35,-2.35)--(1.35,-.75);
\path[draw] (1.35,2.35)--(1.35,.75);
\path[draw] (4.65,-2.35)--(4.65,-.75);
\path[draw] (4.65,2.35)--(4.65,.75);

\path[draw] (.15,1.5)--(2.35,1.5);
\path[draw] (3.65,1.5)--(5.85,1.5);
\path[draw] (.15,-1.5)--(2.35,-1.5);
\path[draw] (3.65,-1.5)--(5.85,-1.5);

\node at (0.5,2) {$a_{1,1}$};
\node at (2.25,2) {$a_{1,4}$};
\node at (0.5,.75) {$a_{1,2}$};
\node at (2.25,.75) {$a_{1,3}$};

\node at (0.5,-2.25) {$a_{2,2}$};
\node at (2.25,-2.25) {$a_{2,3}$};
\node at (0.5,-1) {$a_{2,1}$};
\node at (2.25,-1) {$a_{2,4}$};

\node at (3.85,2) {$a_{4,1}$};
\node at (5.5,2) {$a_{4,4}$};
\node at (3.85,.75) {$a_{4,2}$};
\node at (5.5,.75) {$a_{4,3}$};

\node at (3.85,-2.25) {$a_{3,2}$};
\node at (5.5,-2.25) {$a_{3,3}$};
\node at (3.85,-1) {$a_{3,1}$};
\node at (5.5,-1) {$a_{3,4}$};

\draw[->] (7,0)--(9,0);
\node at (8,1) {$P$};

\path[draw] (10,0)--(16,0);
\path[draw] (13,-3)--(13,3);
\path[draw] (11.35,-2.35)--(11.35,-.75);
\path[draw] (11.35,2.35)--(11.35,.75);
\path[draw] (14.65,-2.35)--(14.65,-.75);
\path[draw] (14.65,2.35)--(14.65,.75);

\path[draw] (10.15,1.5)--(12.35,1.5);
\path[draw] (13.65,1.5)--(15.85,1.5);
\path[draw] (10.15,-1.5)--(12.35,-1.5);
\path[draw] (13.65,-1.5)--(15.85,-1.5);

\node at (10.5,2) {$a_{1,1}$};
\node at (12.25,2) {$a_{1,2}$};
\node at (10.5,.75) {$a_{1,4}$};
\node at (12.25,.75) {$a_{1,3}$};

\node at (10.5,-2.25) {$a_{4,4}$};
\node at (12.25,-2.25) {$a_{4,3}$};
\node at (10.5,-1) {$a_{4,1}$};
\node at (12.25,-1) {$a_{4,2}$};

\node at (13.85,2) {$a_{2,1}$};
\node at (15.5,2) {$a_{2,2}$};
\node at (13.85,.75) {$a_{2,4}$};
\node at (15.5,.75) {$a_{2,3}$};

\node at (13.85,-2.25) {$a_{3,4}$};
\node at (15.5,-2.25) {$a_{3,3}$};
\node at (13.85,-1) {$a_{3,1}$};
\node at (15.5,-1) {$a_{3,2}$};
\end{tikzpicture}
\end{center}
\caption{Reflections of a $2\times 2$ super tic-tac-toe board}\label{ref_board_mult}
\end{figure}

\noindent Notice that we can can do any combination of rotations and reflections. By reflecting, rotating, reflecting and rotating, we end up back at the original game board: 

\begin{figure}[h]
\begin{center}
\begin{tikzpicture}[scale=.45]
\path[draw] (0,0)--(6,0);
\path[draw] (3,-3)--(3,3);
\path[draw] (1.35,-2.35)--(1.35,-.75);
\path[draw] (1.35,2.35)--(1.35,.75);
\path[draw] (4.65,-2.35)--(4.65,-.75);
\path[draw] (4.65,2.35)--(4.65,.75);

\path[draw] (.15,1.5)--(2.35,1.5);
\path[draw] (3.65,1.5)--(5.85,1.5);
\path[draw] (.15,-1.5)--(2.35,-1.5);
\path[draw] (3.65,-1.5)--(5.85,-1.5);

\node at (0.5,2) {$a_{1,1}$};
\node at (2.25,2) {$a_{1,4}$};
\node at (0.5,.75) {$a_{1,2}$};
\node at (2.25,.75) {$a_{1,3}$};

\node at (0.5,-2.25) {$a_{2,2}$};
\node at (2.25,-2.25) {$a_{2,3}$};
\node at (0.5,-1) {$a_{2,1}$};
\node at (2.25,-1) {$a_{2,4}$};

\node at (3.85,2) {$a_{4,1}$};
\node at (5.5,2) {$a_{4,4}$};
\node at (3.85,.75) {$a_{4,2}$};
\node at (5.5,.75) {$a_{4,3}$};

\node at (3.85,-2.25) {$a_{3,2}$};
\node at (5.5,-2.25) {$a_{3,3}$};
\node at (3.85,-1) {$a_{3,1}$};
\node at (5.5,-1) {$a_{3,4}$};


\draw[->] (7,0)--(9,0);
\node at (8,1) {$\Sigma P$};

\path[draw] (10,0)--(16,0);
\path[draw] (13,-3)--(13,3);
\path[draw] (11.35,-2.35)--(11.35,-.75);
\path[draw] (11.35,2.35)--(11.35,.75);
\path[draw] (14.65,-2.35)--(14.65,-.75);
\path[draw] (14.65,2.35)--(14.65,.75);

\path[draw] (10.15,1.5)--(12.35,1.5);
\path[draw] (13.65,1.5)--(15.85,1.5);
\path[draw] (10.15,-1.5)--(12.35,-1.5);
\path[draw] (13.65,-1.5)--(15.85,-1.5);

\node at (10.5,2) {$a_{2,2}$};
\node at (12.25,2) {$a_{2,3}$};
\node at (10.5,.75) {$a_{2,1}$};
\node at (12.25,.75) {$a_{2,4}$};

\node at (10.5,-2.25) {$a_{1,1}$};
\node at (12.25,-2.25) {$a_{1,4}$};
\node at (10.5,-1) {$a_{1,2}$};
\node at (12.25,-1) {$a_{1,3}$};

\node at (13.85,2) {$a_{3,2}$};
\node at (15.5,2) {$a_{3,3}$};
\node at (13.85,.75) {$a_{3,1}$};
\node at (15.5,.75) {$a_{3,4}$};

\node at (13.85,-2.25) {$a_{4,1}$};
\node at (15.5,-2.25) {$a_{4,4}$};
\node at (13.85,-1) {$a_{4,2}$};
\node at (15.5,-1) {$a_{4,3}$};

\draw[->] (17,0)--(19,0);
\node at (18,1) {$\Sigma P$};

\path[draw] (20,0)--(26,0);
\path[draw] (23,-3)--(23,3);
\path[draw] (21.35,-2.35)--(21.35,-.75);
\path[draw] (21.35,2.35)--(21.35,.75);
\path[draw] (24.65,-2.35)--(24.65,-.75);
\path[draw] (24.65,2.35)--(24.65,.75);

\path[draw] (20.15,1.5)--(22.35,1.5);
\path[draw] (23.65,1.5)--(25.85,1.5);
\path[draw] (20.15,-1.5)--(22.35,-1.5);
\path[draw] (23.65,-1.5)--(25.85,-1.5);

\node at (20.5,2) {$a_{1,1}$};
\node at (22.25,2) {$a_{1,4}$};
\node at (20.5,.75) {$a_{1,2}$};
\node at (22.25,.75) {$a_{1,3}$};

\node at (20.5,-2.25) {$a_{2,2}$};
\node at (22.25,-2.25) {$a_{2,3}$};
\node at (20.5,-1) {$a_{2,1}$};
\node at (22.25,-1) {$a_{2,4}$};

\node at (23.85,2) {$a_{4,1}$};
\node at (25.5,2) {$a_{4,4}$};
\node at (23.85,.75) {$a_{4,2}$};
\node at (25.5,.75) {$a_{4,3}$};

\node at (23.85,-2.25) {$a_{3,2}$};
\node at (25.5,-2.25) {$a_{3,3}$};
\node at (23.85,-1) {$a_{3,1}$};
\node at (25.5,-1) {$a_{3,4}$};
\end{tikzpicture}
\end{center}
\caption{A rotation and reflection repeated twice of a $2\times 2$ super tic-tac-toe board}\label{rot_ref}
\end{figure}
\noindent We conclude that the group generated by $D_4=\langle \Sigma, P| ~~(\Sigma P)^2=\Sigma^4=P^2=1 \rangle$ acts on the $2\times 2$ game boards of super tic-tac-toe.

\end{example}

We are now ready to state our main result. 
\begin{theorem}\label{dihedral}
The dihedral group $D_m$ acts on the set of $n$-boards $\mathcal{B}_n$ where 

\begin{itemize}
\item $m$=LCM$(4n-4, 4n-12, 4n-20,\cdots, 4)$, if $n$ is even. 
\item $m$=LCM$(4n-4, 4n-12,\cdots 8, 1)$, if $n$ is odd.
\end{itemize}
\end{theorem}

\begin{proof} Let $[a_{1,1},\dots, a_{1,n^2}, a_{2,1}, \dots a_{2,n^2}, \dots a_{n^2,1}, \dots a_{n^2, n^2}]$ represent $B_n\in\mathcal{B}_n$ and let $\Sigma$ and $P$ be defined as above.

Let $m$ be the order of the rotation $\Sigma$. Then $m$ is the  least common multiple of the orders of each cycle $\sigma_k$ for $k\in [1,l]$, where $l=\displaystyle{\left\lfloor\dfrac{n+1}{2}\right\rfloor}$. If $n$ is even, then the order of $\sigma_k$ is $4+8(k-1)$. If $n$ is odd, then the order of $\sigma_k$ is  

$$o(\sigma_k)=\begin{cases}
1, &k=1\\
8(k-1)& k\neq 1.
\end{cases}$$

The order of the reflection $P$ is 2 since $P$ is a product of transpositions. If we show that $(\Sigma P)^2=1$, then we've proved our theorem since $D_m$ has a presentation given by $\langle \Sigma,P| \Sigma^m=P^2=(\Sigma P)^2=1 \rangle$. We can write:

\begin{align}\label{equation}(\Sigma P)^2([a_{1,1},\dots, &a_{1,n^2}, a_{2,1}, \dots a_{2,n^2}, \dots a_{n^2,1}, \dots a_{n^2, n^2}])\\
&=[a_{(\sigma\rho)^2(1),(\sigma\rho)^2(1)}, \dots, a_{(\sigma\rho)^2(n^2), (\sigma\rho)^2(n^2)}].\nonumber\end{align}

Now, \begin{align*}(\sigma \rho)^2=(\sigma_1\sigma_2\sigma_3\cdots\sigma_l\rho_1\rho_2\rho_3\cdots\rho_l)^2=(\sigma_1   \rho_1 \sigma_1\rho_1)\dots (\sigma_l   \rho_l \sigma_l\rho_l),\end{align*} 

\noindent since $\sigma_k$ and  $\rho_j$ are disjoint, if $k\neq j$. 

\noindent Now, Suppose $\alpha$ is in layer $k$ of a field (or board). Then, 

\begin{align*}
(\sigma\rho)^2\alpha&=\sigma_k\rho_k\sigma_k\rho_k(\alpha)\\&=\sigma_k\rho_k\sigma_k(-\alpha)\\&=\sigma_k\rho_k(-\alpha+1)\\&=\sigma_k(\alpha-1)\\&=\alpha. \\
\end{align*}

\noindent Thus can rewrite Equation (\ref{equation}) as 

\begin{align*}(\Sigma P)^2([a_{1,1},\dots, &a_{1,n^2}, a_{2,1}, \dots a_{2,n^2}, \dots a_{n^2,1}, \dots a_{n^2, n^2}])\\ 
&=[a_{(\sigma\rho)^2(1),(\sigma\rho)^2(1)}, \dots , a_{(\sigma\rho)^2(n^2),(\sigma\rho)^2(n^2)}]\\
&=[a_{1,1},\dots, a_{1,n^2}, a_{2,1}, \dots a_{2,n^2}, \dots a_{n^2,1}, \dots a_{n^2, n^2}].
\end{align*}

Therefore, $(\Sigma P)^2=1$ proving the result. 	
\end{proof}

\subsection{Group Actions of Super Tic-Tac-Toe Games}

A game, $g$, of impartial super tic-tac-toe can be represented by a list $g=[\alpha_{i_1,j_1}, \alpha_{i_2,j_2}, \dots, \alpha_{i_{r},j_{r}}]$, where $\alpha_{i_l, j_l}$ is the position of the $l^{th}$ move in the game. Recall that the current player's move determines the $i$-Field that the next player is allowed to play in. Consider the $l^{th}$ move of the game, $\alpha_{i_l,j_l}$. By the rules of our game, $j_l$ will dictate $i_{l+1}$. That is, if $F_{j_l}$ is not already won, then $i_{l+1}=j_l$. If $F_{j_l}$ is already won, then $i_{l+1} \neq j_l$.

We define the group action of games in a similar manner to the group action on $n$-boards. Let $g=[\alpha_{i_1,j_1}, \alpha_{i_2,j_2}, \dots, \alpha_{i_{r},j_{r}}]$ represent a game of super tic-tac-toe and $G$ be the set of all games. A {\it rotation} of a game, $g$ is given by: $$\Sigma([\alpha_{i_1,j_1}, \alpha_{i_2,j_2}, \dots, \alpha_{i_r,j_r}])=[\alpha_{\sigma(i_1),\sigma(j_1)}, \alpha_{\sigma(i_2),\sigma(j_2)}, \dots, \alpha_{\sigma(i_r),\sigma(j_r)}].$$

A {\it reflection} of a game $g$ is given by: $$P([\alpha_{i_1,j_1}, \alpha_{i_2,j_2}, \dots, \alpha_{i_r,j_r}])=[\alpha_{\rho(i_1),\rho(j_1)}, \alpha_{\rho(i_2),\rho(j_2)}, \dots, \alpha_{\rho(i_r),\rho(j_r)}].$$

We need to check that the game structure and rules will hold for STTT boards under the group actions. Since we rotate or reflect both the game board as well as the $i$-fields, it is easy to see that the $l-1$ move will still dictate the correct $i$-field for the $l^{th}$ move when there are possible moves in that $i$-field. It remains to show that if the $i$-field for the $l^{th}$ move is won, then a move is available in another $i$-field that is not affected by a rotation or reflection.

\begin{theorem}Games are preserved under the $D_m$ group action on STTT boards. 
\end{theorem}

\begin{proof} Using the STTT rules, we need to verify that $\Sigma([\alpha_{i_1,j_1}, \alpha_{i_2,j_2}, \dots, \alpha_{i_r,j_r}])$ and 

\noindent$P([\alpha_{i_1,j_1}, \alpha_{i_2,j_2}, \dots, \alpha_{i_r,j_r}])$ are valid game boards. Suppose we have a game, 

\noindent$g=[\alpha_{i_1,j_1}, \alpha_{i_2,j_2}, \dots, \alpha_{i_r,j_r}]$ whose first $l-1$  moves are into fields with available positions, but whose $l^{th}$ move is into a field that has already been filled. Then we have $i_{l+1}\neq j_{l}$. Consider the rotation of $g$: 

$$\Sigma([\alpha_{i_1,j_1}, \alpha_{i_2,j_2}, \dots, \alpha_{i_r,j_r}])=[\alpha_{\sigma(i_1),\sigma(j_1)}, \alpha_{\sigma(i_2),\sigma(j_2)}, \dots, \alpha_{\sigma(i_r),\sigma(j_r)}]$$

\noindent We have $\sigma(j_m)=\sigma(i_{m+1})$ for $m\in[1,l-1]$. So, the game structure is preserved for the first $l-1$ moves. However, $\sigma(j_{l})\neq\sigma(i_{l+1})$. We need to verify that this $l^{th}$ move is allowed. In terms of game play, this means that if a player has played in position $\alpha_{\sigma(i_{l}), \sigma(j_{l})}$, then the next player can play in position $\alpha_{\sigma(i_{l+1}), \sigma(j_{l+1})}$. 

Suppose not. Then, this means that $\alpha_{\sigma(i_{l+1}), \sigma(j_{l+1})}$ has already been played. Therefore, there exists a $p\in[1,l-1]$ so that $\alpha_{\sigma(i_{l+1}), \sigma(j_{l+1})}=\alpha_{\sigma(i_{p}), \sigma(j_{p})}$. Hence $(\sigma(i_{l+1}),\sigma(j_{l+1}))=(\sigma(i_{p}),\sigma(j_{p}))\Rightarrow l+1=p$. However, this is a contradiction since $p<l+1$. Therefore, $\Sigma([\alpha_{i_1,j_1}, \alpha_{i_2,j_2}, \dots, \alpha_{i_r,j_r}])\in G$. A similar argument holds for $P$. 
\end{proof}

\section{Applications and Further Questions}
The $D_m$ action on games and boards allows us to group together games which have similar structures. For a game $g=[\alpha_{i_1,j_1}, \alpha_{i_2,j_2}, \dots, \alpha_{i_r,j_r}]$ in an $n\times n$ super tic-tac-toe board, we are able to keep track of the order in which the plays are made. If we {\it only} look at the winning game, where an $X$ has been played in each of the positions $\alpha_{i_1,j_1}, \alpha_{i_2,j_2}, \dots, \alpha_{i_r,j_r}$, then it is not always possible to determine the order in which the plays were made.

\begin{example}\label{example}
Let's consider the game $[\alpha_{3,1},\alpha_{1,1},\alpha_{1,3},\alpha_{3,3}]$ played on a $2\times 2$ super tic-tac-toe board. The rotation $\sigma=(1,2,3,4)$ and reflection $\rho=(2,4)$ act on the $2\times 2$ square and we have a corresponding group action by $D_4$ on our $2\times 2$ super tic-tac-toe games. The orbit of $[\alpha_{3,1},\alpha_{1,1},\alpha_{1,3},\alpha_{3,3}]$ consists of

$$\{[\alpha_{3,1},\alpha_{1,1},\alpha_{1,3},\alpha_{3,3}],~[\alpha_{1,3},\alpha_{3,3},\alpha_{3,1},\alpha_{1,1}],~[\alpha_{4,2},\alpha_{2,2},\alpha_{2,4},\alpha_{4,4}],~[\alpha_{2,4},\alpha_{4,4},\alpha_{4,2},\alpha_{2,2}]\}.$$

If we {\it only} care about the structure of the winning game board and not the order in which the game is played, then we can make the observation that both $[\alpha_{3,1},\alpha_{1,1},\alpha_{1,3},\alpha_{3,3}]$ and $[\alpha_{1,3},\alpha_{3,3},\alpha_{3,1},\alpha_{1,1}]$ represent the game to the left of Figure \ref{order2}. Similarly, the games $[\alpha_{4,2},\alpha_{2,2},\alpha_{2,4},\alpha_{4,4}]$ and $[\alpha_{2,4},\alpha_{4,4},\alpha_{4,2},\alpha_{2,2}]$ represent the game to the right of Figure \ref{order2}. 

\begin{figure}[h!]
\begin{center}\includegraphics[scale=.35]{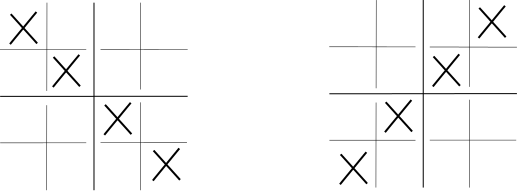}\end{center}
\caption{Two games on a $2\times 2$ board}%
\label{order2}%
\end{figure} 
\end{example}  

We say that two games $g_1$ and $g_2$ on an $n\times n$ super tic-tac-toe board are {\it isomorphic} if there exists an element $\gamma\in D_m$ so that $\gamma (g_1)=g_2$. In the last example, we are able to observe that there are 4 elements in the isomorphism class of $[\alpha_{3,1},\alpha_{1,1},\alpha_{1,3},\alpha_{3,3}]$.

If we forget about the order in which the $X$'s are played and only look at the structure of a winning game, we can talk about isomorphism classes of winning {\it game boards}. 

\begin{definition} Let $\bar{g}\in\mathcal{B}_n$ be the game board for a winning game $g$ in $n\times n$ super tic-tac-toe. The game boards $\bar{g}$ and $\bar{g'}$ are isomorphic if $\exists~ \gamma\in D_m$ so that $\gamma(\bar{g})=\bar{g'}$. We say that $\bar{g}$ and $\bar{g'}$ belong to the same isomorphism class. 
\end{definition} 

\noindent The two game boards in Figure \ref{order2} represent the isomorphism class for the game board for $g=[\alpha_{3,1},\alpha_{1,1},\alpha_{1,3},\alpha_{3,3}]$.   If we continue to look at winning game boards for $2\times 2$ super tic-tac-toe, we find the following: 

\begin{itemize}\item There is 1 isomorphism class of game boards with 2 elements. This is precisely the class represented in Figure \ref{order2}. 

\item There are 19 isomorphism classes of game boards with 4 elements. 

\item There are 228 isomorphism classes of game boards with 8 elements. 
\end{itemize}

\noindent A complete list of winning $2\times 2$ super tic-tac-toe game board isomorphism classes was developed through a python code and can be found in Appendix A. 

\begin{figure}[h!]
\begin{center}\includegraphics[scale=.35]{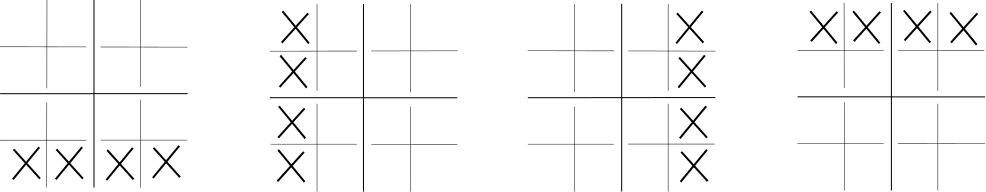}\end{center}
\caption{An example of an isomorphism class with 4 elements  for a $2\times 2$ board}%
\label{order4}%
\end{figure}

\begin{figure}[h!]
\begin{center}\includegraphics[scale=.35]{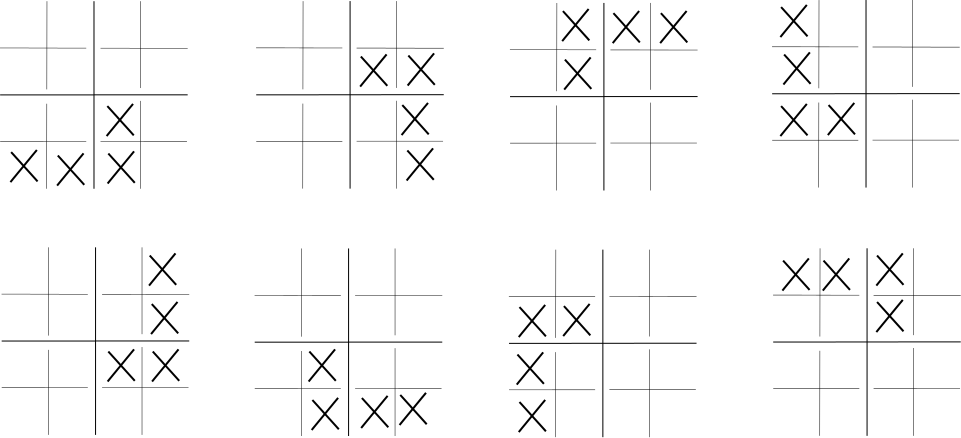}\end{center}
\caption{An example of an isomorphism class with 8 elements  for a $2\times 2$ board}%
\label{order8}%
\end{figure}

\subsection{Questions}
It is natural to try to extend finding isomorphism classes to larger super tic-tac-toe boards. Once this is done, it would be interesting to see if there are any patterns in how the size of the isomorphism classes grow as $n$ grows. We could also look at how the number of winning games with a certain number, say the largest number of possible moves to get a winning game, changes as $n$ increases.

Another interesting direction to take this work is towards other variations of tic-tac-toe. For instance, could a three dimensional tic-tac-toe game support a ``super'' version? In the regular $3\times 3\times 3$ case, the first player can easily when by playing in the middle. Would the same hold for a super $3\times 3\times 3$ game of tic-tac-toe? In Qubic, a $4\times 4\times 4$ version of tic-tac-toe, player 1 can force a win \cite{oren}. Is this still true for ``Super Qubic''? Tic-tac-toe has been played on a torus. What about super tic-tac-toe?

\appendix
\section{Isomorphism Classes for $2\times 2$ STTT Boards}

There is 1 isomorphism class of order 2, 1 isomorphism classes of order 4 and 228 isomorphism classes of order 8. Below are the isomorphism classes. Note that in each list, a zero in the $a_{i,j}$ spot corresponds to nothing in the $a_{i,j}$ position in the $2$-board and a 1 in the $a_{i,j}$ spot corresponds to an $x$ in the $a_{i,j}$ position in the $2$-board where $$(a_{1,1} a_{1,2} a_{1,3} a_{1,4} a_{2,1} a_{2,2} a_{2,3} a_{2,4} a_{3,1} a_{3,2} a_{3,3} a_{3,4} a_{4,1} a_{4,2} a_{4,3} a_{4,4})$$ corresponds to the board:

\begin{figure}[h]
\begin{center}
\begin{tikzpicture}[scale=.5]
\path[draw] (0,0)--(6,0);
\path[draw] (3,-3)--(3,3);
\path[draw] (1.35,-2.35)--(1.35,-.75);
\path[draw] (1.35,2.35)--(1.35,.75);
\path[draw] (4.65,-2.35)--(4.65,-.75);
\path[draw] (4.65,2.35)--(4.65,.75);

\path[draw] (.15,1.5)--(2.35,1.5);
\path[draw] (3.65,1.5)--(5.85,1.5);
\path[draw] (.15,-1.5)--(2.35,-1.5);
\path[draw] (3.65,-1.5)--(5.85,-1.5);

\node at (0.5,2) {$a_{1,1}$};
\node at (2.25,2) {$a_{1,2}$};
\node at (0.5,.75) {$a_{1,3}$};
\node at (2.25,.75) {$a_{1,4}$};

\node at (0.5,-2.25) {$a_{3,3}$};
\node at (2.25,-2.25) {$a_{3,4}$};
\node at (0.5,-1) {$a_{3,1}$};
\node at (2.25,-1) {$a_{3,2}$};

\node at (3.85,2) {$a_{2,1}$};
\node at (5.5,2) {$a_{2,2}$};
\node at (3.85,.75) {$a_{2,3}$};
\node at (5.5,.75) {$a_{2,4}$};

\node at (3.85,-2.25) {$a_{4,3}$};
\node at (5.5,-2.25) {$a_{4,4}$};
\node at (3.85,-1) {$a_{4,1}$};
\node at (5.5,-1) {$a_{4,2}$};
\end{tikzpicture}
\caption{$2\times2$ STTT Board}\label{2STTT}
\end{center}
\end{figure}

\noindent{\underline{{\it{Isomorphism of Order 2: }}}}
\begin{enumerate}
	\item (0000011001100000, 1001000000001001)
\end{enumerate}

\noindent{\underline{\it{Isomorphism of Order 4: }}}
\begin{enumerate}
\item (0000000000110011, 0000010100000101, 1010000010100000, 1100110000000000) 
	
\item (0010000100110011, 0100010100010101, 1010100010100010, 1100110010000100)
	
\item (0010000110100101, 0100110000010011, 1010010110000100, 1100100000110010) 
	
\item (0010010110100100, 0100110000110010, 1010000110000101, 1100100000010011) 
	
\item (0001001000110011, 0001010101000101, 1010001010101000, 1100110001001000)
	
\item (0001001010100101, 0001110001000011, 1010010101001000, 1100001000111000) 
	
\item (0010110000110100, 0100010110100010, 1010100000010101, 1100000110000011) 
	
\item (0010011001100010, 0100011001100100, 1001000100011001, 1001100010001001) 
	
\item (0010011001100100, 0100011001100010, 1001000110001001, 1001100000011001) 
	
\item (0001001001101001, 0001011001001001, 1001001001101000, 1001011001001000) 
	
\item (0010100101100010, 0100011010010100, 0110000100011001, 1001100010000110) 
	
\item (0010000101101001, 0100011000011001, 1001011010000100, 1001100001100010) 
	
\item (0010011010010010, 0100100101100100, 0110100010001001, 1001000100010110) 
	
\item (0010000110010110, 0100100100010110, 0110100010010010, 0110100110000100) 
	
\item (0001001001011010, 0001001101001100, 0011001011001000, 0101101001001000)
 
\item (0010001111000100, 0011000110001100, 0100101001010010, 0101100000011010) 
	
\item (0010000101011010, 0011100011000010, 0100001100011100, 0101101010000100) 
	
\item (0010101001010100, 0011100000011100, 0100001111000010, 0101000110001010) 
	
\item (0001001010010110, 0001100101000110, 0110001010011000, 0110100101001000)
\end{enumerate}

\noindent{\underline{\it{Isomorphism of Order 8: }}}

\begin{enumerate}
\item (0000000000111010, 0000000001010011, 0000001100000101, 0000010100001100, 

0011000010100000, 0101110000000000, 1010000011000000, 1100101000000000)

\item (0000000000110110, 0000000010010011, 0000010100000110, 0000100100000101,

 0110000010100000, 0110110000000000, 1010000010010000, 1100100100000000) 

\item (0000000100110011, 0000010100010101, 0010000000110011, 0100010100000101,

 1010000010100010, 1010100010100000, 1100110000000100, 1100110010000000) 

\item (0000000101010011, 0000001100010101, 0010000000111010, 0011000010100010,

 0100010100001100, 0101110000000100, 1010100011000000, 1100101010000000) 

\item (0000000110010011, 0000100100010101, 0010000000110110, 0100010100000110,

 0110000010100010, 0110110000000100, 1010100010010000, 1100100110000000) 

\item (0000000101010101, 0000001100010011, 0010000010101010, 0011000000110010,

 0100110000001100, 0101010100000100, 1010101010000000, 1100100011000000) 

\item (0000000100110101, 0000010100010011, 0010000010100011, 0100110000000101, 

1010000000110010, 1010110010000000, 1100010100000100, 1100100010100000) 

\item (0000000110010101, 0000100100010011, 0010000010100110, 0100110000000110,

 0110000000110010, 0110010100000100, 1010100110000000, 1100100010010000) 

\item (0001000100110011, 0001010100010101, 0010001000110011, 0100010101000101,

 1010001010100010, 1010100010101000, 1100110001000100, 1100110010001000) 

\item (0001000101010011, 0001001100010101, 0010001000111010, 0011001010100010,

 0100010101001100, 0101110001000100, 1010100011001000, 1100101010001000) 

\item (0001000110010011, 0001100100010101, 0010001000110110, 0100010101000110,

 0110001010100010, 0110110001000100, 1010100010011000, 1100100110001000) 

\item (0001000101010101, 0001001100010011, 0010001010101010, 0011001000110010, 

0100110001001100, 0101010101000100, 1010101010001000, 1100100011001000) 

\item (0001000100110101, 0001010100010011, 0010001010100011, 0100110001000101,

 1010001000110010, 1010110010001000, 1100010101000100, 1100100010101000) 

\item (0001000110010101, 0001100100010011, 0010001010100110, 0100110001000110,

 0110001000110010, 0110010101000100, 1010100110001000, 1100100010011000) 

\item (0010001100110010, 0010101010100010, 0011000100010011, 0100010101010100,

 0100110011000100, 0101000100010101, 1010100010001010, 1100100010001100) 

\item (0010001110100010, 0010101000110010, 0011000100010101, 0100010111000100, 

0100110001010100, 0101000100010011, 1010100010001100, 1100100010001010) 

\item (0010011000110010, 0010011010100010, 0100010101100100, 0100110001100100, 

1001000100010011, 1001000100010101, 1010100010001001, 1100100010001001) 

\item (0010000100111010, 0010000101010011, 0011100010100010, 0100001100010101,

 0100010100011100, 0101110010000100, 1010100011000010, 1100101010000100) 

\item (0010000100110110, 0010000110010011, 0100010100010110, 0100100100010101,

 0110100010100010, 0110110010000100, 1010100010010010, 1100100110000100) 

\item (0010000101010101, 0010000110101010, 0011100000110010, 0100001100010011,

 0100110000011100, 0101010110000100, 1010101010000100, 1100100011000010) 

\item (0010000100110101, 0010000110100011, 0100010100010011, 0100110000010101, 

1010100000110010, 1010110010000100, 1100010110000100, 1100100010100010) 

\item (0010000110010101, 0010000110100110, 0100100100010011, 0100110000010110, 

0110010110000100, 0110100000110010, 1010100110000100, 1100100010010010) 

\item (0010010100110010, 0010110010100010, 0100010100110100, 0100110010100100,

 1010000100010011, 1010100010000011, 1100000100010101, 1100100010000101) 

\item (0010010110100010, 0010110000110010, 0100010110100100, 0100110000110100, 

1010000100010101, 1010100010000101, 1100000100010011, 1100100010000011) 

\item (0000001000110011, 0000010101000101, 0001000000110011, 0001010100000101,

 1010000010101000, 1010001010100000, 1100110000001000, 1100110001000000) 

\item (0000001001010011, 0000001101000101, 0001000000111010, 0001010100001100,

 0011000010101000, 0101110000001000, 1010001011000000, 1100101001000000) 

\item (0000001010010011, 0000100101000101, 0001000000110110, 0001010100000110,

 0110000010101000, 0110110000001000, 1010001010010000, 1100100101000000) 

\item (0000000101100101, 0000011000010011, 0010000010101001, 0100110000001001,

 1001000000110010, 1001010100000100, 1010011010000000, 1100100001100000) 

\item (0000000110100101, 0000110000010011, 0010000010100101, 0100110000000011,

 1010010100000100, 1010010110000000, 1100000000110010, 1100100000110000) 

\item (0001010110000101, 0001100000110011, 0010010101000101, 0100001000110011, 

1010000110101000, 1010001010100100, 1100110000011000, 1100110001000010) 

\item (0001001110000101, 0001100001010011, 0010010101001100, 0011001010100100,

 0100001000111010, 0101110001000010, 1010000111001000, 1100101000011000) 

\item (0001100010010011, 0001100110000101, 0010010101000110, 0100001000110110,

 0110001010100100, 0110110001000010, 1010000110011000, 1100100100011000) 

\item (0001000111000101, 0001101000010011, 0010001010101100, 0011010101000100,

 0100110001001010, 0101001000110010, 1010001110001000, 1100100001011000) 

\item (0001000110100101, 0001110000010011, 0010001010100101, 0100110001000011, 

1010010101000100, 1010010110001000, 1100001000110010, 1100100000111000) 

\item (0010010101010100, 0010101010100100, 0011100000010011, 0100001100110010,

 0100110011000010, 0101000110000101, 1010000110001010, 1100100000011100) 

\item (0010001110100100, 0010010111000100, 0011000110000101, 0100101000110010,

 0100110001010010, 0101100000010011, 1010000110001100, 1100100000011010) 

\item (0010010101100100, 0010011010100100, 0100011000110010, 0100110001100010,

 1001000110000101, 1001100000010011, 1010000110001001, 1100100000011001) 

\item (0010010100010101, 0010100000110011, 0100000100110011, 0100010110000101,

 1010000110100010, 1010100010100100, 1100110000010100, 1100110010000010) 

\item (0010010100011100, 0010100001010011, 0011100010100100, 0100000100111010,

 0100001110000101, 0101110010000010, 1010000111000010, 1100101000010100) 

\item (0010010100010110, 0010100010010011, 0100000100110110, 0100100110000101,

 0110100010100100, 0110110010000010, 1010000110010010, 1100100100010100) 

\item (0010000110101100, 0010000111000101, 0011010110000100, 0100101000010011,

 0100110000011010, 0101100000110010, 1010001110000100, 1100100001010010) 

\item (0010010100110100, 0010110010100100, 0100010100110010, 0100110010100010,

 1010000110000011, 1010100000010011, 1100000110000101, 1100100000010101) 

\item (0000001000111010, 0000010101001100, 0001000001010011, 0001001100000101,

 0011001010100000, 0101110001000000, 1010000011001000, 1100101000001000) 

\item (0000001000110110, 0000010101000110, 0001000010010011, 0001100100000101,

 0110001010100000, 0110110001000000, 1010000010011000, 1100100100001000) 

\item (0001001000111010, 0001001001010011, 0001001101000101, 0001010101001100,

 0011001010101000, 0101110001001000, 1010001011001000, 1100101001001000) 

\item (0001001000110110, 0001001010010011, 0001010101000110, 0001100101000101, 

0110001010101000, 0110110001001000, 1010001010011000, 1100100101001000) 

\item (0001000101100101, 0001011000010011, 0010001010101001, 0100110001001001,

 1001001000110010, 1001010101000100, 1010011010001000, 1100100001101000) 

\item (0000001100110010, 0000010101010100, 0010101010100000, 0011000000010011,

 0100110011000000, 0101000100000101, 1010000010001010, 1100100000001100) 

\item (0000010111000100, 0000101000110010, 0010001110100000, 0011000100000101, 

0100110001010000, 0101000000010011, 1010000010001100, 1100100000001010) 

\item (0000010101100100, 0000011000110010, 0010011010100000, 0100110001100000,

 1001000000010011, 1001000100000101, 1010000010001001, 1100100000001001) 

\item (0000000100111010, 0000010100011100, 0010000001010011, 0011100010100000,

 0100001100000101, 0101110010000000, 1010000011000010, 1100101000000100) 

\item (0000000100110110, 0000010100010110, 0010000010010011, 0100100100000101,

 0110100010100000, 0110110010000000, 1010000010010010, 1100100100000100) 

\item (0010001100010101, 0010100000111010, 0011000110100010, 0100000101010011,

 0100010110001100, 0101110000010100, 1010100011000100, 1100101010000010) 

\item (0010100000110110, 0010100100010101, 0100000110010011, 0100010110000110,

 0110000110100010, 0110110000010100, 1010100010010100, 1100100110000010) 

\item (0010100100110010, 0010100110100010, 0100010110010100, 0100110010010100,

 0110000100010011, 0110000100010101, 1010100010000110, 1100100010000110) 

\item (0001010110001100, 0001100000111010, 0010001101000101, 0011000110101000,

 0100001001010011, 0101110000011000, 1010001011000100, 1100101001000010) 

\item (0001010110000110, 0001100000110110, 0010100101000101, 0100001010010011,

 0110000110101000, 0110110000011000, 1010001010010100, 1100100101000010) 

\item (0010000101100101, 0010000110101001, 0100011000010011, 0100110000011001, 

1001010110000100, 1001100000110010, 1010011010000100, 1100100001100010) 

\item (0010010110010100, 0010100110100100, 0100100100110010, 0100110010010010,

 0110000110000101, 0110100000010011, 1010000110000110, 1100100000010110) 

\item (0000010100110010, 0000010100110100, 0010110010100000, 0100110010100000, 

1010000000010011, 1010000010000011, 1100000100000101, 1100100000000101) 

\item (0000010110100100, 0000110000110010, 0010010110100000, 0100110000110000, 

1010000010000101, 1010000100000101, 1100000000010011, 1100100000000011) 

\item (0000000000111001, 0000000001100011, 0000010100001001, 0000011000000101, 

1001000010100000, 1001110000000000, 1010000001100000, 1100011000000000) 

\item (0000000000110101, 0000000010100011, 0000010100000011, 0000110000000101,

 1010000000110000, 1010110000000000, 1100000010100000, 1100010100000000) 

\item (0000000101100011, 0000011000010101, 0010000000111001, 0100010100001001,

 1001000010100010, 1001110000000100, 1010100001100000, 1100011010000000) 

\item (0000000111000011, 0000101000010101, 0010000000111100, 0011110000000100,

 0100010100001010, 0101000010100010, 1010100001010000, 1100001110000000) 

\item (0000001001010101, 0000001101000011, 0001000010101010, 0001110000001100,

 0011000000111000, 0101010100001000, 1010101001000000, 1100001011000000) 

\item (0000001000110101, 0000010101000011, 0001000010100011, 0001110000000101,

 1010000000111000, 1010110001000000, 1100001010100000, 1100010100001000) 

\item (0000001010010101, 0000100101000011, 0001000010100110, 0001110000000110,

 0110000000111000, 0110010100001000, 1010100101000000, 1100001010010000) 

\item (0001000101100011, 0001011000010101, 0010001000111001, 0100010101001001,

 1001001010100010, 1001110001000100, 1010100001101000, 1100011010001000) 

\item (0001000111000011, 0001101000010101, 0010001000111100, 0011110001000100,

 0100010101001010, 0101001010100010, 1010100001011000, 1100001110001000) 

\item (0001001001010101, 0001001010101010, 0001001101000011, 0001110001001100,

 0011001000111000, 0101010101001000, 1010101001001000, 1100001011001000) 

\item (0001001000110101, 0001001010100011, 0001010101000011, 0001110001000101,

 1010001000111000, 1010110001001000, 1100001010101000, 1100010101001000) 

\item (0001001010010101, 0001001010100110, 0001100101000011, 0001110001000110,

 0110001000111000, 0110010101001000, 1010100101001000, 1100001010011000) 

\item (0001101010100010, 0001110011000100, 0010001100111000, 0011000101000011,

 0100010101011000, 0101001000010101, 1010100001001010, 1100001010001100) 

\item (0001001110100010, 0001110001010100, 0010101000111000, 0011001000010101,

 0100010111001000, 0101000101000011, 1010100001001100, 1100001010001010) 

\item (0001011010100010, 0001110001100100, 0010011000111000, 0100010101101000,

 1001000101000011, 1001001000010101, 1010100001001001, 1100001010001001) 

\item (0010000100111001, 0010000101100011, 0100010100011001, 0100011000010101,

 1001100010100010, 1001110010000100, 1010100001100010, 1100011010000100) 

\item (0010000100111100, 0010000111000011, 0011110010000100, 0100010100011010,

 0100101000010101, 0101100010100010, 1010100001010010, 1100001110000100) 

\item (0001000110101010, 0001110000011100, 0010001001010101, 0011100000111000,

 0100001101000011, 0101010110001000, 1010101001000100, 1100001011000010) 

\item (0001000110100011, 0001110000010101, 0010001000110101, 0100010101000011,

 1010100000111000, 1010110001000100, 1100001010100010, 1100010110001000) 

\item (0001000110100110, 0001110000010110, 0010001010010101, 0100100101000011,

 0110010110001000, 0110100000111000, 1010100101000100, 1100001010010010) 

\item (0001110010100010, 0001110010100100, 0010010100111000, 0100010100111000,

 1010000101000011, 1010100001000011, 1100001000010101, 1100001010000101) 

\item (0001010110100010, 0001110000110100, 0010110000111000, 0100010110101000,

 1010001000010101, 1010100001000101, 1100000101000011, 1100001010000011)

\item (0000001001100011, 0000011001000101, 0001000000111001, 0001010100001001, 

1001000010101000, 1001110000001000, 1010001001100000, 1100011001000000) 

\item (0000001011000011, 0000101001000101, 0001000000111100, 0001010100001010,

 0011110000001000, 0101000010101000, 1010001001010000, 1100001101000000) 

\item (0000001001100101, 0000011001000011, 0001000010101001, 0001110000001001,

 1001000000111000, 1001010100001000, 1010011001000000, 1100001001100000) 

\item (0000001010100101, 0000110001000011, 0001000010100101, 0001110000000011,

 1010010100001000, 1010010101000000, 1100000000111000, 1100001000110000) 

\item (0001011010000101, 0001100001100011, 0010010101001001, 0100001000111001,

 1001001010100100, 1001110001000010, 1010000101101000, 1100011000011000) 

\item (0001100011000011, 0001101010000101, 0010010101001010, 0011110001000010,

 0100001000111100, 0101001010100100, 1010000101011000, 1100001100011000) 

\item (0001001010101100, 0001001011000101, 0001101001000011, 0001110001001010,

 0011010101001000, 0101001000111000, 1010001101001000, 1100001001011000) 

\item (0001101010100100, 0001110011000010, 0010010101011000, 0011100001000011,

 0100001100111000, 0101001010000101, 1010000101001010, 1100001000011100) 

\item (0001001110100100, 0001110001010010, 0010010111001000, 0011001010000101,

 0100101000111000, 0101100001000011, 1010000101001100, 1100001000011010) 

\item (0001011010100100, 0001110001100010, 0010010101101000, 0100011000111000,

 1001001010000101, 1001100001000011, 1010000101001001, 1100001000011001) 

\item (0010010100011001, 0010100001100011, 0100000100111001, 0100011010000101, 

1001100010100100, 1001110010000010, 1010000101100010, 1100011000010100) 

\item (0010010100011010, 0010100011000011, 0011110010000010, 0100000100111100,

 0100101010000101, 0101100010100100, 1010000101010010, 1100001100010100) 

\item (0001000110101100, 0001110000011010, 0010001011000101, 0011010110001000,

 0100101001000011, 0101100000111000, 1010001101000100, 1100001001010010) 

\item (0001010110100100, 0001110000110010, 0010010110101000, 0100110000111000,

 1010000101000101, 1010001010000101, 1100001000010011, 1100100001000011) 

\item (0000001000111100, 0000010101001010, 0001000011000011, 0001101000000101,

 0011110001000000, 0101001010100000, 1010000001011000, 1100001100001000) 

\item (0001001110000011, 0001100001010101, 0010110001001100, 0011001000110100,

 0100001010101010, 0101010101000010, 1010101000011000, 1100000111001000) 

\item (0001010110000011, 0001100000110101, 0010110001000101, 0100001010100011, 

1010001000110100, 1010110000011000, 1100000110101000, 1100010101000010) 

\item (0001100010010101, 0001100110000011, 0010110001000110, 0100001010100110,

 0110001000110100, 0110010101000010, 1010100100011000, 1100000110011000) 

\item (0010001100110100, 0010110011000100, 0011000110000011, 0100010101010010, 

0100101010100010, 0101100000010101, 1010100000011010, 1100000110001100) 

\item (0010101000110100, 0010110001010100, 0011100000010101, 0100001110100010,

 0100010111000010, 0101000110000011, 1010100000011100, 1100000110001010) 

\item (0010011000110100, 0010110001100100, 0100010101100010, 0100011010100010,

 1001000110000011, 1001100000010101, 1010100000011001, 1100000110001001) 

\item (0001001000111100, 0001001011000011, 0001010101001010, 0001101001000101, 

0011110001001000, 0101001010101000, 1010001001011000, 1100001101001000) 

\item (0001011010000011, 0001100001100101, 0010110001001001, 0100001010101001,

 1001001000110100, 1001010101000010, 1010011000011000, 1100000101101000) 

\item (0001100010100101, 0001110010000011, 0010110001000011, 0100001010100101,

 1010010100011000, 1010010101000010, 1100000100111000, 1100001000110100) 

\item (0010010101010010, 0010110011000010, 0011100010000011, 0100001100110100, 

0100101010100100, 0101100010000101, 1010000100011010, 1100000100011100) 

\item (0010010111000010, 0010110001010010, 0011100010000101, 0100001110100100,

 0100101000110100, 0101100010000011, 1010000100011100, 1100000100011010) 

\item (0010010101100010, 0010110001100010, 0100011000110100, 0100011010100100,

 1001100010000011, 1001100010000101, 1010000100011001, 1100000100011001) 

\item (0000001100110100, 0000010101010010, 0010110011000000, 0011000010000011,

 0100101010100000, 0101100000000101, 1010000000011010, 1100000100001100) 

\item (0000010111000010, 0000101000110100, 0010110001010000, 0011100000000101, 

0100001110100000, 0101000010000011, 1010000000011100, 1100000100001010) 

\item (0000010101100010, 0000011000110100, 0010110001100000, 0100011010100000,
 
1001000010000011, 1001100000000101, 1010000000011001, 1100000100001001) 

\item (0000000100111100, 0000010100011010, 0010000011000011, 0011110010000000,

 0100101000000101, 0101100010100000, 1010000001010010, 1100001100000100) 

\item (0010100000110101, 0010110000010101, 0100000110100011, 0100010110000011, 

1010100000110100, 1010110000010100, 1100000110100010, 1100010110000010) 

\item (0010100000111100, 0010101000010101, 0011110000010100, 0100000111000011, 

0100010110001010, 0101000110100010, 1010100001010100, 1100001110000010) 

\item (0010100001010101, 0010110000011100, 0011100000110100, 0100000110101010,

 0100001110000011, 0101010110000010, 1010101000010100, 1100000111000010) 

\item (0010100010010101, 0010110000010110, 0100000110100110, 0100100110000011,

 0110010110000010, 0110100000110100, 1010100100010100, 1100000110010010) 

\item (0010100100110100, 0010110010010100, 0100010110010010, 0100100110100010, 

0110000110000011, 0110100000010101, 1010100000010110, 1100000110000110) 

\item (0001010110001010, 0001100000111100, 0010101001000101, 0011110000011000,

 0100001011000011, 0101000110101000, 1010001001010100, 1100001101000010) 

\item (0010100001100101, 0010110000011001, 0100000110101001, 0100011010000011, 

1001010110000010, 1001100000110100, 1010011000010100, 1100000101100010) 

\item (0010100010100101, 0010110000010011, 0100000110100101, 0100110010000011,

 1010010100010100, 1010010110000010, 1100000100110010, 1100100000110100) 

\item (0010010110010010, 0010110010010010, 0100100100110100, 0100100110100100,

 0110100010000011, 0110100010000101, 1010000100010110, 1100000100010110) 

\item (0000010110100010, 0000110000110100, 0010110000110000, 0100010110100000, 

1010000000010101, 1010100000000101, 1100000010000011, 1100000100000011) 

\item (0000001001101010, 0000011001001100, 0001000001011001, 0001001100001001,

 0011001001100000, 0101011001000000, 1001000011001000, 1001101000001000) 

\item (0000001001100110, 0000011001000110, 0001000010011001, 0001100100001001, 

0110001001100000, 0110011001000000, 1001000010011000, 1001100100001000) 

\item (0001000100111001, 0001010100011001, 0010001001100011, 0100011001000101,

 1001100010101000, 1001110010001000, 1010001001100010, 1100011001000100) 

\item (0001000101011001, 0001001100011001, 0010001001101010, 0011001001100010,

 0100011001001100, 0101011001000100, 1001100011001000, 1001101010001000) 

\item (0001000110011001, 0001100100011001, 0010001001100110, 0100011001000110, 

0110001001100010, 0110011001000100, 1001100010011000, 1001100110001000) 

\item (0010001101100010, 0010101001100010, 0011000100011001, 0100011001010100,

 0100011011000100, 0101000100011001, 1001100010001010, 1001100010001100) 

\item (0001001000111001, 0001001001100011, 0001010101001001, 0001011001000101,

 1001001010101000, 1001110001001000, 1010001001101000, 1100011001001000) 

\item (0001001001011001, 0001001001101010, 0001001101001001, 0001011001001100, 

0011001001101000, 0101011001001000, 1001001011001000, 1001101001001000) 

\item (0001001001100110, 0001001010011001, 0001011001000110, 0001100101001001,

 0110001001101000, 0110011001001000, 1001001010011000, 1001100101001000) 

\item (0001000101101001, 0001011000011001, 0010001001101001, 0100011001001001,

 1001001001100010, 1001011001000100, 1001011010001000, 1001100001101000)

\item (0001000110101001, 0001110000011001, 0010001001100101, 0100011001000011, 

1001010110001000, 1001100000111000, 1010011001000100, 1100001001100010) 

\item (0010011001010100, 0010101001100100, 0011100000011001, 0100001101100010,

 0100011011000010, 0101000110001001, 1001000110001010, 1001100000011100) 

\item (0010001101100100, 0010011011000100, 0011000110001001, 0100011001010010,

 0100101001100010, 0101100000011001, 1001000110001100, 1001100000011010) 

\item (0000001101100010, 0000011001010100, 0010101001100000, 0011000000011001,

 0100011011000000, 0101000100001001, 1001000010001010, 1001100000001100) 

\item (0000011011000100, 0000101001100010, 0010001101100000, 0011000100001001,

 0100011001010000, 0101000000011001, 1001000010001100, 1001100000001010) 

\item (0000011001100010, 0000011001100100, 0010011001100000, 0100011001100000,

 1001000000011001, 1001000010001001, 1001000100001001, 1001100000001001) 

\item (0000001001101001, 0000011001001001, 0001000001101001, 0001011000001001,

 1001000001101000, 1001001001100000, 1001011000001000, 1001011001000000) 

\item (0001000111001001, 0001101000011001, 0010001001101100, 0011011001000100,

 0100011001001010, 0101001001100010, 1001001110001000, 1001100001011000) 

\item (0001011011000100, 0001101001100010, 0010001101101000, 0011000101001001,

 0100011001011000, 0101001000011001, 1001001010001100, 1001100001001010) 

\item (0001001101100010, 0001011001010100, 0010101001101000, 0011001000011001,

 0100011011001000, 0101000101001001, 1001001010001010, 1001100001001100) 

\item (0001011001100010, 0001011001100100, 0010011001101000, 0100011001101000,

 1001000101001001, 1001001000011001, 1001001010001001, 1001100001001001) 

\item (0001001001101100, 0001001011001001, 0001011001001010, 0001101001001001,

 0011011001001000, 0101001001101000, 1001001001011000, 1001001101001000) 

\item (0001001001100101, 0001001010101001, 0001011001000011, 0001110001001001,

 1001001000111000, 1001010101001000, 1010011001001000, 1100001001101000) 

\item (0001011011000010, 0001101001100100, 0010011001011000, 0011100001001001,

 0100001101101000, 0101001010001001, 1001000101001010, 1001001000011100) 

\item (0001001101100100, 0001011001010010, 0010011011001000, 0011001010001001, 

0100101001101000, 0101100001001001, 1001000101001100, 1001001000011010) 

\item (0000001101100100, 0000011001010010, 0010011011000000, 0011000010001001,

 0100101001100000, 0101100000001001, 1001000000011010, 1001000100001100) 

\item (0000011011000010, 0000101001100100, 0010011001010000, 0011100000001001,

 0100001101100000, 0101000010001001, 1001000000011100, 1001000100001010) 

\item (0000001101100000, 0000011001010000, 0000011011000000, 0000101001100000,

 0011000000001001, 0101000000001001, 1001000000001010, 1001000000001100) 

\item (0000000101101010, 0000011000011100, 0010000001011001, 0011100001100000,

 0100001100001001, 0101011010000000, 1001000011000010, 1001101000000100) 

\item (0000000101100110, 0000011000010110, 0010000010011001, 0100100100001001,

 0110011010000000, 0110100001100000, 1001000010010010, 1001100100000100) 

\item (0010000101011001, 0010000101101010, 0011100001100010, 0100001100011001,

 0100011000011100, 0101011010000100, 1001100011000010, 1001101010000100) 

\item (0010000101100110, 0010000110011001, 0100011000010110, 0100100100011001,

 0110011010000100, 0110100001100010, 1001100010010010, 1001100110000100) 

\item (0010001100011001, 0010100001101010, 0011000101100010, 0100000101011001, 

0100011010001100, 0101011000010100, 1001100011000100, 1001101010000010) 

\item (0010100001100110, 0010100100011001, 0100000110011001, 0100011010000110,

 0110000101100010, 0110011000010100, 1001100010010100, 1001100110000010) 

\item (0001000101101010, 0001011000011100, 0010001001011001, 0011100001101000,

 0100001101001001, 0101011010001000, 1001001011000010, 1001101001000100) 

\item (0001000101100110, 0001011000010110, 0010001010011001, 0100100101001001,

 0110011010001000, 0110100001101000, 1001001010010010, 1001100101000100) 

\item (0010011000011001, 0010100001101001, 0100000101101001, 0100011010001001,

 1001000101100010, 1001011000010100, 1001011010000010, 1001100001100100) 

\item (0010011010010100, 0010100101100100, 0100011010010010, 0100100101100010,

 0110000110001001, 0110100000011001, 1001000110000110, 1001100000010110) 

\item (0000011010010100, 0000100101100010, 0010100101100000, 0100011010010000,

 0110000000011001, 0110000100001001, 1001000010000110, 1001100000000110) 

\item (0000000101101001, 0000011000011001, 0010000001101001, 0100011000001001,

 1001000001100010, 1001011000000100, 1001011010000000, 1001100001100000) 

\item (0010000101101100, 0010000111001001, 0011011010000100, 0100011000011010,

 0100101000011001, 0101100001100010, 1001001110000100, 1001100001010010) 

\item (0001011010001100, 0001100001101010, 0010001101001001, 0011000101101000,

 0100001001011001, 0101011000011000, 1001001011000100, 1001101001000010) 

\item (0001011010000110, 0001100001100110, 0010100101001001, 0100001010011001,

 0110000101101000, 0110011000011000, 1001001010010100, 1001100101000010) 

\item (0001011010010100, 0001100101100010, 0010100101101000, 0100011010011000,

 0110000101001001, 0110001000011001, 1001001010000110, 1001100001000110) 

\item (0001000101101100, 0001011000011010, 0010001011001001, 0011011010001000, 

0100101001001001, 0101100001101000, 1001001001010010, 1001001101000100) 

\item (0001011010001001, 0001100001101001, 0010011001001001, 0100001001101001, 

1001000101101000, 1001001001100100, 1001011000011000, 1001011001000010) 

\item (0001011010010010, 0001100101100100, 0010011010011000, 0100100101101000,

 0110001010001001, 0110100001001001, 1001000101000110, 1001001000010110) 

\item (0000011010010010, 0000100101100100, 0010011010010000, 0100100101100000,

 0110000010001001, 0110100000001001, 1001000000010110, 1001000100000110) 

\item (0000010101100000, 0000011000110000, 0000011010100000, 0000110001100000,

 1001000000000011, 1001000000000101, 1010000000001001, 1100000000001001) 

\item (0000000101010110, 0000001100010110, 0010000010011010, 0011000010010010,

 0100100100001100, 0101100100000100, 0110100011000000, 0110101010000000) 

\item (0000000110010110, 0000100100010110, 0010000010010110, 0100100100000110, 

0110000010010010, 0110100010010000, 0110100100000100, 0110100110000000) 

\item (0001000100110110, 0001010100010110, 0010001010010011, 0100100101000101, 

0110100010101000, 0110110010001000, 1010001010010010, 1100100101000100) 

\item (0001000101010110, 0001001100010110, 0010001010011010, 0011001010010010,

 0100100101001100, 0101100101000100, 0110100011001000, 0110101010001000) 

\item (0001000110010110, 0001100100010110, 0010001010010110, 0100100101000110,

 0110001010010010, 0110100010011000, 0110100101000100, 0110100110001000) 

\item (0010001110010010, 0010101010010010, 0011000100010110, 0100100101010100, 

0100100111000100, 0101000100010110, 0110100010001010, 0110100010001100) 

\item (0010000101010110, 0010000110011010, 0011100010010010, 0100001100010110,

 0100100100011100, 0101100110000100, 0110100011000010, 0110101010000100) 

\item (0000001001010110, 0000001101000110, 0001000010011010, 0001100100001100,

 0011000010011000, 0101100100001000, 0110001011000000, 0110101001000000) 

\item (0000001010010110, 0000100101000110, 0001000010010110, 0001100100000110,

 0110000010011000, 0110001010010000, 0110100100001000, 0110100101000000) 

\item (0000000110100110, 0000110000010110, 0010000010010101, 0100100100000011, 

0110010110000000, 0110100000110000, 1010100100000100, 1100000010010010) 

\item (0001001110000110, 0001100001010110, 0010100101001100, 0011001010010100,

 0100001010011010, 0101100101000010, 0110000111001000, 0110101000011000) 

\item (0001100010010110, 0001100110000110, 0010100101000110, 0100001010010110,

 0110000110011000, 0110001010010100, 0110100100011000, 0110100101000010) 

\item (0001000111000110, 0001101000010110, 0010001010011100, 0011100101000100, 

0100100101001010, 0101001010010010, 0110001110001000, 0110100001011000) 

\item (0010100101010100, 0010101010010100, 0011100000010110, 0100001110010010,

 0100100111000010, 0101000110000110, 0110000110001010, 0110100000011100) 

\item (0010001110010100, 0010100111000100, 0011000110000110, 0100100101010010,

 0100101010010010, 0101100000010110, 0110000110001100, 0110100000011010) 

\item (0010100001010110, 0010100100011100, 0011100010010100, 0100000110011010,

 0100001110000110, 0101100110000010, 0110000111000010, 0110101000010100) 

\item (0010100010010110, 0010100100010110, 0100000110010110, 0100100110000110, 

0110000110010010, 0110100010010100, 0110100100010100, 0110100110000010) 

\item (0010000110011100, 0010000111000110, 0011100110000100, 0100100100011010,

 0100101000010110, 0101100010010010, 0110001110000100, 0110100001010010) 

\item (0000001001011010, 0000001101001100, 0001000001011010, 0001001100001100,

 0011000011001000, 0011001011000000, 0101101000001000, 0101101001000000) 

\item (0001000100111010, 0001010100011100, 0010001001010011, 0011100010101000,

 0100001101000101, 0101110010001000, 1010001011000010, 1100101001000100) 

\item (0001000101011010, 0001001100011100, 0010001001011010, 0011001011000010,

 0011100011001000, 0100001101001100, 0101101001000100, 0101101010001000) 

\item (0001000110011010, 0001100100011100, 0010001001010110, 0011100010011000,

 0100001101000110, 0101100110001000, 0110001011000010, 0110101001000100) 

\item (0001000101011100, 0001001100011010, 0010001011001010, 0011001001010010,

 0011101010001000, 0100101001001100, 0101001101000100, 0101100011001000) 

\item (0001000100111100, 0001010100011010, 0010001011000011, 0011110010001000,

 0100101001000101, 0101100010101000, 1010001001010010, 1100001101000100) 

\item (0001000110011100, 0001100100011010, 0010001011000110, 0011100110001000,

 0100101001000110, 0101100010011000, 0110001001010010, 0110001101000100) 

\item (0010001101010010, 0010101011000010, 0011000100011010, 0011100010001010, 

0100001101010100, 0100101011000100, 0101000100011100, 0101100010001100) 

\item (0010001111000010, 0010101001010010, 0011000100011100, 0011100010001100,

 0100001111000100, 0100101001010100, 0101000100011010, 0101100010001010) 

\item (0010011001010010, 0010011011000010, 0011100010001001, 0100001101100100,

 0100101001100100, 0101100010001001, 1001000100011010, 1001000100011100) 

\item (0001001001010110, 0001001010011010, 0001001101000110, 0001100101001100,

 0011001010011000, 0101100101001000, 0110001011001000, 0110101001001000) 

\item (0010001101010100, 0010101011000100, 0011000110001010, 0011100000011010, 

0100001101010010, 0100101011000010, 0101000110001100, 0101100000011100) 

\item (0000001101010010, 0000001101010100, 0010101011000000, 0011000000011010,

 0011000010001010, 0100101011000000, 0101000100001100, 0101100000001100) 

\item (0000001111000100, 0000101001010010, 0010001111000000, 0011000010001100,

 0011000100001100, 0100101001010000, 0101000000011010, 0101100000001010) 

\item (0000000101011010, 0000001100011100, 0010000001011010, 0011000011000010,

 0011100011000000, 0100001100001100, 0101101000000100, 0101101010000000) 

\item (0010001100011100, 0010100001011010, 0011000111000010, 0011100011000100,

 0100000101011010, 0100001110001100, 0101101000010100, 0101101010000010) 

\item (0010000101011100, 0010000111001010, 0011100001010010, 0011101010000100,

 0100001100011010, 0100101000011100, 0101001110000100, 0101100011000010) 

\item (0010100101010010, 0010100111000010, 0011100010000110, 0100001110010100,

 0100101010010100, 0101100010000110, 0110000100011010, 0110000100011100) 

\item (0001001110001100, 0001100001011010, 0010001101001100, 0011000111001000, 

0011001011000100, 0100001001011010, 0101101000011000, 0101101001000010) 

\item (0000001110100100, 0000110001010010, 0010010111000000, 0011000010000101, 

0100101000110000, 0101100000000011, 1010000100001100, 1100000000011010) 

\item (0000011010100100, 0000110001100010, 0010010101100000, 0100011000110000,

 1001000010000101, 1001100000000011, 1010000100001001, 1100000000011001) 

\item (0000000110100011, 0000110000010101, 0010000000110101, 0100010100000011,

 1010100000110000, 1010110000000100, 1100000010100010, 1100010110000000) 

\item (0001000111001010, 0001101000011100, 0010001001011100, 0011100001011000,

 0011101001000100, 0100001101001010, 0101001011000010, 0101001110001000) 

\item (0001001001011100, 0001001011001010, 0001001101001010, 0001101001001100,

 0011001001011000, 0011101001001000, 0101001011001000, 0101001101001000) 

\item (0001001010011100, 0001001011000110, 0001100101001010, 0001101001000110, 

0011100101001000, 0101001010011000, 0110001001011000, 0110001101001000) 

\item (0001101011000010, 0001101011000100, 0010001101011000, 0011000101001010, 

0011100001001010, 0100001101011000, 0101001000011100, 0101001010001100) 

\item (0001001111000010, 0001101001010100, 0010101001011000, 0011001000011100,

 0011100001001100, 0100001111001000, 0101000101001010, 0101001010001010) 

\item (0001010111000010, 0001101000110100, 0010110001011000, 0011100001000101,

 0100001110101000, 0101001010000011, 1010001000011100, 1100000101001010) 

\item (0001100010100011, 0001110010000101, 0010010101000011, 0100001000110101,

 1010000100111000, 1010110001000010, 1100001010100100, 1100010100011000) 

\item (0001001111000100, 0001101001010010, 0010001111001000, 0011000101001100,

 0011001010001100, 0100101001011000, 0101001000011010, 0101100001001010) 

\item (0001100010100110, 0001110010000110, 0010100101000011, 0100001010010101,

 0110000100111000, 0110010100011000, 1010100101000010, 1100001010010100) 

\item (0000001111000010, 0000101001010100, 0010101001010000, 0011000000011100,

 0011100000001100, 0100001111000000, 0101000010001010, 0101000100001010) 

\item (0010011000010101, 0010100000111001, 0100000101100011, 0100010110001001,

 1001000110100010, 1001110000010100, 1010100001100100, 1100011010000010) 

\item (0000001001011001, 0000001101001001, 0001000001101010, 0001011000001100,

 0011000001101000, 0101011000001000, 1001001011000000, 1001101001000000) 

\item (0010011000011100, 0010100001011001, 0011100001100100, 0100000101101010,

 0100001110001001, 0101011010000010, 1001000111000010, 1001101000010100) 

\item (0001010101100010, 0001011000110100, 0010110001101000, 0100011010101000,

 1001001010000011, 1001100001000101, 1010001000011001, 1100000101001001) 

\item (0000001000111001, 0000010101001001, 0001000001100011, 0001011000000101,

 1001001010100000, 1001110001000000, 1010000001101000, 1100011000001000)

\end{enumerate}

    \bibliographystyle{plain}
    \bibliography{bibliography}

\begin{thebibliography}{1}

\bibitem{carroll}
Maureen~T. {Carroll} and Steven~T. {Dougherty}.
\newblock {Tic-Tac-Toe on a Finite Plane}.
\newblock {\em Mathematics Magazine}, 77(4):260--274, 2004.

\bibitem{goff}
Allan Goff.
\newblock {Quantum tic-tac-toe: A teaching metaphor for superposition in
  quantum mechanics}.
\newblock {\em American Journal of Physics}, 74(11):962--973, 2006.

\bibitem{2by2}
Kamini {Janardan} and Chris {Prijic}.
\newblock {Analysis of a Contemporary $2\times 2$ Super Tic-Tac-Toe Board},
  2012.

\bibitem{oren}
Oren Patashnik.
\newblock {Qubic: $4\times 4 \times 4$ Tic-Tac-Toe}.
\newblock {\em Mathematics Magazine}, 53(4):202--216, 1980.

\bibitem{notakto}
Thane~E. {Plambeck} and Greg {Whitehead}.
\newblock {The Secrets of Notakto: Winning at X-only Tic-Tac-Toe}.
\newblock {\em ArXiv e-prints}, January 2013.

\end{thebibliography}

\end{document}